\documentclass{amsart}

\usepackage{amsmath,amssymb,amsthm,amscd}
\usepackage{pinlabel}

\newtheorem{prop}{Proposition}[section]
\newtheorem{thm}[prop]{Theorem}
\newtheorem{lem}[prop]{Lemma}
\newtheorem{cor}[prop]{Corollary}

\theoremstyle{definition}

\newtheorem{defn}[prop]{Definition}
\newtheorem{ex}[prop]{Example}
\newtheorem{exs}[prop]{Examples}
\newtheorem{rem}[prop]{Remark}

\newtheorem*{ack}{Acknowledgements}

\numberwithin{equation}{section}

\def\co{\colon\thinspace}
\newcommand{\ip}{{\,\rule{2.3mm}{.2mm}\rule{.2mm}{2.3mm}\;\, }}

\newcommand{\oa}{\overline{a}}
\newcommand{\alphac}{\alpha_{\mathrm{c}}}
\newcommand{\oalphac}{\overline{\alpha}_{\mathrm{c}}}

\newcommand{\C}{\mathbb{C}}
\newcommand{\CC}{\mathcal{C}}

\newcommand{\D}{\mathbb{D}}
\newcommand{\DD}{\mathcal{D}}
\newcommand{\rmd}{\mathrm{d}}

\newcommand{\rme}{\mathrm{e}}
\newcommand{\EE}{\mathcal{E}}

\newcommand{\FF}{\mathcal{F}}

\newcommand{\rmi}{\mathrm{i}}

\newcommand{\blambda}{\mbox{\boldmath{$\lambda$}}}

\newcommand{\MM}{\mathcal{M}}

\newcommand{\bm}{\mathbf{m}}

\newcommand{\N}{\mathbb{N}}

\newcommand{\oomega}{\overline{\omega}}
\newcommand{\omegac}{\omega_{\mathrm{c}}}
\newcommand{\oomegac}{\overline{\omega}_{\mathrm{c}}}

\newcommand{\Q}{\mathbb{Q}}

\newcommand{\PP}{\mathcal{P}}
\newcommand{\ophi}{\overline{\phi}}

\newcommand{\R}{\mathbb{R}}

\newcommand{\Z}{\mathbb{Z}}
\newcommand{\oz}{\overline{z}}

\DeclareMathOperator{\Int}{Int}

\DeclareMathOperator{\re}{Re}

\DeclareMathOperator{\sign}{sign}

\begin{document}

\author[H.~Geiges]{Hansj\"org Geiges}
\address{Mathematisches Institut, Universit\"at zu K\"oln,
Weyertal 86--90, 50931 K\"oln, Germany}
\email{geiges@math.uni-koeln.de}
\author[J.~Gonzalo P\'erez]{Jes\'us Gonzalo P\'erez}
\address{Departamento de Matem\'aticas, Universidad Aut\'onoma
de Madrid, 28049 Madrid, Spain}
\email{jesus.gonzalo@uam.es}

\title[Transversely holomorphic flows]{Transversely holomorphic flows
and contact circles on spherical $3$-manifolds}

\date{}

\begin{abstract}
Motivated by the moduli theory of taut contact circles on
spherical $3$-manifolds, we relate taut contact circles
to transversely holomorphic flows. We give an elementary survey
of such $1$-dimensional foliations from a topological viewpoint. We
describe a complex analogue of the classical Godbillon--Vey invariant,
the so-called Bott invariant, and a logarithmic monodromy of closed leaves.
The Bott invariant allows us to formulate a generalised Gau{\ss}--Bonnet
theorem. We compute these invariants for the Poincar\'e foliations on the
$3$-sphere and derive rigidity statements, including
a uniformisation theorem for orbifolds. These results are then
applied to the classification of taut contact circles.
\end{abstract}


\thanks{J. G. is partially supported by grant MTM2014-57769-C3-3-P from
MINECO Spain.}
\maketitle

%%%%%%%%%%%%%%%%%%%%%%%%%%%%%%%%%%%%%%%%%%%%%%%%%%%%%%%%%%%%%%%%%%%%%%

\section{Introduction}
Transversely holomorphic flows on $3$-manifolds have been
classified by Brunel\-la~\cite{brun96} and Ghys~\cite{ghys96}.
The taut contact circles (Definition~\ref{defn:tcc})
studied by us in a series of papers
beginning with \cite{gego95} are special instances of such
transversely holomorphic flows. Indeed, the classification
in \cite{brun96} of $3$-manifolds that admit a transversely holomorphic
flow follows a route
via the Enriques--Kodaira classification of complex surfaces
similar to the one taken in~\cite{gego95}.

In \cite{gego02} we indicated that the moduli theory of
taut contact circles on spherical $3$-manifolds admits
a nice reformulation in terms of an invariant for
transversely holomorphic flows, which, it turns out,
is the basic incarnation of a secondary characteristic class
first constructed by Bott~\cite{bott72}.

In order to develop this moduli theory in a way accessible to contact
geometers, we present in this paper a detailed survey of transversely
holomorphic flows (or oriented $1$-dimensional foliations) on $3$-manifolds,
notably on the $3$-sphere $S^3$. For it is only on manifolds covered
by $S^3$ that this moduli problem is linked in an intriguing fashion with
the common kernel foliation of the taut contact circle.

We describe the construction of the Bott class
(Definition~\ref{defn:GV}), a global
invariant for transversely holomorphic flows, as a direct complex
analogue of the Godbillon--Vey invariant~\cite{gove71}. We also
introduce a logarithmic monodromy for closed leaves in such foliations
(Definition~\ref{defn:log-m}), which can be interpreted as
a simple instance of the residue theory for transversely holomorphic
foliations developed by Heitsch~\cite{heit78}.
We use the Bott invariant to formulate
a generalised Gau{\ss}--Bonnet theorem (Theorem~\ref{thm:GV-real}),
from which we deduce the classical Gau{\ss}--Bonnet theorem
in Corollary~\ref{cor:GB}.

Motivated by the moduli problem for taut contact circles~\cite{gego02},
we then turn our attention to transversely holomorphic foliations
on the $3$-sphere $S^3$; these are the so-called
Poincar\'e foliations of~\cite{brun96}.
The Bott invariant turns out to be the moduli
parameter in each of two families of taut contact circles.

We give explicit models for the transversely holomorphic foliations
on $S^3$ and show this list to be exhaustive (Theorem~\ref{thm:all-P})
by appealing to the Poincar\'e--Dulac normalisation theorem
for Poincar\'e singularities.
We compute the Bott invariant of these
foliations, and the logarithmic monodromy of their closed
leaves.

Section~\ref{section:topology} is devoted to a detailed study of the
topology of transversely holomorphic foliations on~$S^3$.
With the aid of associated $2$-dimensional foliations
we provide means to visualise these foliations. This includes
an analysis of the asymptotic behaviour of the non-compact leaves,
and the Poincar\'e return map of compact ones.
The figures in Section~\ref{section:topology} give an inkling
of the rich geometry and dynamics displayed by transversely
holomorphic foliations.

The calculations of the invariants from Sections
\ref{section:S3} and~\ref{section:log}, together with some information gained
from the explicit descriptions of the Poincar\'e foliations
in Section~\ref{section:topology}, are then used to prove a number
of rigidity results, for instance about the uniqueness of
the transverse holomorphic structure (Theorem~\ref{thm:unique-hol}).
Within the realm of taut contact circles, we show that
the classification can be given in terms of the
common kernel foliation (Theorem~\ref{thm:common-kernel}).
An application of these rigidity results is a uniformisation
theorem for orbifolds (Theorem~\ref{thm:orbi}), which has been
proved previously using the Ricci flow.

In the case where the transversely holomorphic foliation defines
a Seifert fibration, we determine the Seifert invariants
explicitly (Proposition~\ref{prop:Seifert}). In the context
of the rigidity results, 
we make an observation about Seifert fibrations
of $S^3$ and lens spaces (Proposition~\ref{prop:Seifert-S-L})
that may be of independent interest.

Much of what we say about transversely holomorphic flows
on $3$-manifolds, except probably for the generalised
Gau{\ss}--Bonnet theorem and the explicit analysis
of the Poincar\'e foliations, can be found in some form in the specialist
literature. We hope that our survey of the relevant material will not
only make this paper self-contained from a contact geometric
perspective, but also serve as an introduction to the
beautiful theory of transversely holomorphic flows for a wider
audience.
\section{Transversely holomorphic flows and taut contact circles}
\label{section:thf-tcc}
Let $Y$ be a nowhere zero vector field on a closed, oriented $3$-manifold~$M$.
The flow (or the foliation) generated by $Y$ is said to be
\emph{transversely holomorphic} if there is a complex structure $J$
on the $2$-plane bundle $TM/\langle Y\rangle$ invariant under the
flow of~$Y$. This is equivalent to having a transverse conformal
structure and a transverse orientation.

We shall restrict attention to the case where the bundle $TM/\langle Y\rangle$
is trivial. For the study of transversely holomorphic flows on the
$3$-sphere this is no restriction.
Given any nowhere zero vector field $Y$ with this triviality condition,
one can find a pair of pointwise linearly
independent $1$-forms $\omega_1,\omega_2$ on $M$ whose
common kernel $\ker\omega_1\cap\ker\omega_2$ is spanned by~$Y$,
and such that $\omega_1\wedge\omega_2$ defines the transverse
orientation. We introduce the complex-valued
$1$-form $\omegac:=\omega_1+\rmi\omega_2$, and we write
$L_Y$ for the Lie derivative with respect to~$Y$.

\begin{defn}
\label{defn:transverse}
\begin{itemize}
\item[(C1)] The pair $(\omega_1,\omega_2)$ is said to define a
\emph{transverse conformal structure} for the flow of $Y$ if
there is a real-valued function $f$ on $M$ such that
\[ L_Y(\omega_1\otimes\omega_1+\omega_2\otimes\omega_2)=
f(\omega_1\otimes\omega_1+\omega_2\otimes\omega_2).\]

\item[(C2)] The $1$-form $\omegac$ is said to define a \emph{transverse
holomorphic structure} for the flow of $Y$
if there is a complex-valued function $h$
on $M$ such that $L_Y\omegac=h\omegac$.

\item[(C3)] The $1$-form $\omegac$ is \emph{formally integrable}
if $\omegac\wedge\rmd\omegac=0$.
\end{itemize}
\end{defn}

Condition (C2) is equivalent to our more `naive' definition of
a transverse holomorphic structure above (in the case
where $TM/\langle Y\rangle$ is trivial).
In the situation of (C2), the flow of $Y$ pulls back $\omegac$
to a complex multiple of itself, cf.\ \cite[Lemma~1.5.8]{geig08},
and so the flow preserves the complex structure on $TM/\langle Y\rangle$
defined by the dual basis to $(\omega_1,\omega_2)$; the converse argument
is similar.

Conditions (C1) and (C2)
do not depend on the specific choice of~$Y$.
This means that `transversely conformal resp.\
holomorphic' is really a property of the line field $\langle Y\rangle$
or the foliation it defines. An alternative interpretation of this property,
more common in foliation theory, is that the holonomy pseudogroup of the
foliation consists of biholomorphisms between open subsets of~$\C$.
The terminology `flow' emphasises the fact that these foliations
come with a natural orientation induced from the transverse and the
ambient orientation.

\begin{lem}
Conditions {\rm (C1)} to {\rm (C3)}
are equivalent. A further equivalent condition is:

\begin{itemize}
\item[{\rm (C4)}] The pair $(\omega_1,\omega_2)$ satisfies the identities
\[ \begin{array}{rcr}
\omega_1\wedge\rmd\omega_1 & = & \omega_2\wedge\rmd\omega_2,\\
\omega_1\wedge\rmd\omega_2 & = & -\omega_2\wedge\rmd\omega_1.
\end{array} \]
\end{itemize}
\end{lem}

\begin{proof}
The Cartan formula for the Lie derivative gives $L_Y\omega_j=
Y\ip\rmd\omega_j$, hence $L_Y\omega_j$ annihilates~$Y$. This
implies the existence of smooth functions $a_{ij}$ such that
\begin{eqnarray*}
L_Y\omega_1 & = & a_{11}\omega_1+a_{12}\omega_2,\\
L_Y\omega_2 & = & a_{21}\omega_1+a_{22}\omega_2.
\end{eqnarray*}
We compute
\begin{eqnarray*}
L_Y(\omega_1\otimes\omega_1+\omega_2\otimes\omega_2)
 & = & 2a_{11}\omega_1\otimes\omega_1+2a_{22}\omega_2\otimes\omega_2\\
 &   & \mbox{}+(a_{12}+a_{21})
       (\omega_1\otimes\omega_2+\omega_2\otimes\omega_1).
\end{eqnarray*}
Hence, condition (C1) is equivalent to
\begin{equation}
\label{eqn:C}
\left\{\begin{array}{rcr}
a_{11} & = & a_{22},\\
a_{12} & = & -a_{21}.
\end{array}\right.
\end{equation}
The manifold $M$ being $3$-dimensional, two $3$-forms on $M$ are
equal if and only if they yield the same $2$-form under the
inner product with~$Y$. This inner product transforms
the first equality in (C4) into the second equality in~(\ref{eqn:C}),
and the second into the first. Thus, (C1) and (C4) are equivalent.

The system (\ref{eqn:C}) translates into
\[ L_Y(\omega_1+\rmi\omega_2)=(a_{11}-\rmi a_{12})(\omega_1+\rmi\omega_2).\]
This gives the equivalence between (C1) and (C2).

The equivalence between (C3) and (C4) is trivial to check.
\end{proof}

Recall the following concept from~\cite{gego95}:

\begin{defn}
\label{defn:tcc}
A \emph{taut contact circle} on a $3$-manifold is a pair of contact
forms $(\omega_1,\omega_2)$ such that the $1$-form $\lambda_1\omega_1
+\lambda_2\omega_2$ is a contact form defining the same volume form
for all $(\lambda_1,\lambda_2)\in S^1\subset\R^2$.
\end{defn}

This is equivalent to condition (C4), with the additional
contact requirement $\omega_j\wedge\rmd\omega_j\neq 0$.

In \cite{gego95} it was shown that a taut contact circle on a $3$-manifold
$M$ gives rise to a complex structure on $M\times S^1$. Via
the classification of complex surfaces we arrived at
a complete list of closed $3$-manifolds admitting taut contact
circles:

\begin{thm}
\label{thm:taut-manifolds}
A closed, connected
$3$-manifold $M$ admits a taut contact circle if and only if $M$
is diffeomorphic to a left-quotient of one of the following Lie
groups:
\begin{itemize}
\item[(i)] $\mathrm{SU}(2)$, the universal cover of $\mathrm{SO}(3)$,
\item[(ii)] $\tilde{\mathrm{E}}_2$, the universal
cover of the euclidean group,
\item[(iii)] $\widetilde{\mathrm{SL}}_2$, the universal cover
of $\mathrm{PSL}_2\R$,
\end{itemize}
that is, the universal covers of the
groups of orientation-preserving isometries of the $2$-dimensional
geometries.\qed
\end{thm}

In \cite{gego02} we developed a deformation theory for taut
contact circles, and we determined the corresponding Teichm\"uller
and moduli spaces. Some topological aspects of these
moduli spaces were treated in~\cite{gego12}.
For a comprehensive survey on contact circles see~\cite{gego13}.

One of the aims of this paper will be to apply
results from the theory of transversely holomorphic flows,
which will be surveyed below,
in the special setting of taut contact circles. This will
include a classification of taut contact circles on $S^3$
in terms of the dynamics of its common kernel
foliation. A dynamical characterisation of the general
contact circle property was given in~\cite{gonz08}.
The present paper contains, amongst other things, all the results announced
in \cite{gego02} as to appear under the title
`Transversely conformal flows on $3$-manifolds'.

The class (ii) in Theorem~\ref{thm:taut-manifolds} contains only the five
torus bundles over $S^1$ with periodic monodromy. In class (iii),
the common kernel foliation is always given by the unique Seifert fibration
on the manifold in question. So from the viewpoint of transversely
holomorphic flows, only class (i) can be expected to give rise
to a rich theory. In the discussion of explicit models,
we shall restrict attention to transversely holomorphic foliations
on~$S^3$, but most of what we say extends in a natural way to
the left-quotients.

We end this section with two simple examples illustrating the
relation between transversely holomorphic flows and taut
contact circles, and the issue of the triviality of $TM/\langle Y\rangle$.
Observe that any Seifert fibration admits a transverse holomorphic
structure, given by lifting a holomorphic structure from the quotient
orbifold.

\begin{exs}
(1) The Seifert fibration given by a non-trivial circle bundle
over the $2$-torus defines a transversely holomorphic flow
with a trivial complementary plane bundle, so it can be described
by a formally integrable complex $1$-form~$\omegac$. However, the total
space is of geometric type $\mathrm{Nil}^3$ and does not appear
in the list of Theorem~\ref{thm:taut-manifolds}, so there is no choice
of $\omegac$ corresponding to a taut contact circle.

(2) The obvious Seifert fibration of $S^1\times S^2$ has a non-trivial
complementary plane bundle, so it defines a transversely holomorphic
flow that cannot be defined by a formally integrable complex $1$-form.
\end{exs}
\section{Godbillon--Vey theory and the Bott invariant}
Our aim in this section is to describe an invariant of transversely
holomorphic flows coming from formally integrable complex
$1$-forms. The construction is modelled on the classical
Godbillon--Vey invariant \cite{gove71} for codimension~$1$ foliations,
which we review briefly. This so-called Bott
invariant for transversely holomorphic
flows will then be used to prove a generalised Gau{\ss}--Bonnet
theorem for such flows.
\subsection{The classical Godbillon--Vey invariant}
\label{subsection:classical}
Let $N$ be a manifold of dimension at least~$3$, and $\omega$ a nowhere
zero $1$-form defining an integrable hyperplane field $\ker\omega$,
so that the integral manifolds of this hyperplane field constitute a smooth,
coorientable codimension~$1$ foliation. By the Frobenius integrability
theorem, this is equivalent to requiring $\omega\wedge\rmd\omega=0$.
Computing in a local coframe extending $\omega$, and then using a
partition of unity argument, one finds a $1$-form $\alpha$ on $N$
such that $\rmd\omega=\alpha\wedge\omega$. Then
\[ 0=\rmd^2\omega=\rmd\alpha\wedge\omega-\alpha\wedge\rmd\omega
=\rmd\alpha\wedge\omega-\alpha\wedge\alpha\wedge\omega=
\rmd\alpha\wedge\omega.\]
Arguing as before, we find a $1$-form $\beta$ such that $\rmd\alpha=
\beta\wedge\omega$. This implies
\[ \rmd(\alpha\wedge\rmd\alpha)=\rmd\alpha\wedge\rmd\alpha=
\beta\wedge\omega\wedge\beta\wedge\omega=0,\]
so the $3$-form $\alpha\wedge\rmd\alpha$ defines a de
Rham cohomology class
\[ [\alpha\wedge\rmd\alpha]\in H_{\mathrm{dR}}^3(N).\]
This class depends only on the foliation, not on the choice
of $\omega$ or $\alpha$; in particular, the
coorientation of the foliation
implicit in a choice of $\omega$ plays no role:

\begin{itemize}
\item[(i)]
Given any other $1$-form $\alpha'$ with $\rmd\omega=\alpha'\wedge\omega$,
we have $(\alpha'-\alpha)\wedge\omega=0$, hence $\alpha'-\alpha=
f\omega$ for some smooth function $f$ on~$N$. We then compute
\begin{eqnarray*}
\alpha'\wedge\rmd\alpha'
 & = & (\alpha+f\omega)\wedge (\rmd\alpha+\rmd f\wedge\omega+f\,\rmd\omega)\\
 & = & \alpha\wedge\rmd\alpha-\rmd f\wedge\rmd\omega\\
 & = & \alpha\wedge\rmd\alpha-\rmd(f\,\rmd\omega).
\end{eqnarray*}
\item[(ii)] If $\omega$ is replaced by $\tilde{\omega}=g\omega$
for some smooth nowhere zero function $g$ on~$N$, we compute
\begin{eqnarray*}
\rmd\tilde{\omega}
 & = & \rmd g\wedge\omega+g\,\rmd\omega\\
 & = & \rmd g\wedge\omega+g\alpha\wedge\omega\\
 & = & (g^{-1}\rmd g+\alpha)\wedge\tilde{\omega},
\end{eqnarray*}
so we may take $\tilde{\alpha}:=g^{-1}\rmd g+\alpha$. Then
\[ \tilde{\alpha}\wedge\rmd\tilde{\alpha}=
(g^{-1}\rmd g+\alpha)\wedge\rmd(g^{-1}\rmd g+\alpha)=\alpha\wedge\rmd\alpha-
\rmd(g^{-1}\,\rmd g\wedge\alpha).\]
\end{itemize}

For a nice survey on the Godbillon--Vey invariant and its history
see~\cite{ghys89}.
\subsection{Godbillon--Vey theory for transversely holomorphic flows}
We now mimic this construction for 
transversely holomorphic flows on a closed, connected,
oriented $3$-manifold $M$, with the plane bundle
complementary to the flow being trivial. Any such flow determines
a formally integrable complex $1$-form $\omegac$
(with pointwise linearly independent real and imaginary part),
unique up to multiplication with a nowhere zero, smooth complex-valued
function.

The formal integrability of $\omegac$ gives us a complex $1$-form
$\alphac$ such that
\[ \rmd\omegac=\alphac\wedge\omegac.\]
Computations analogous to (i) and (ii) above,
with $f$ and $g$ complex-valued, show that the cohomology class
$[\alphac\wedge\rmd\alphac]\in H^3_{\mathrm{dR}}(M)\otimes\C\cong\C$
is independent of choices. We interpret this class
as a complex number:

\begin{defn}
\label{defn:GV}
We call the complex number
\[ \int_M\alphac\wedge\rmd\alphac\]
the \emph{Bott invariant} of the transversely
holomorphic flow defined by the formally integrable $1$-form $\omegac$.
\end{defn}

\begin{rem}
In the monograph by Pittie~\cite{pitt76}, this invariant is called
the \emph{complex Godbillon--Vey class}, as one might have expected
from the construction we described. However, we follow
Asuke \cite[Definition~1.1.5]{asuk10} by naming it after Bott.
As explained on page~3 of Asuke's monograph, both for historical reasons
and in order to distinguish it from a different complex generalisation of
the Godbillon--Vey invariant, the attribution to Bott is the
preferred one.

This invariant makes one of its first appearances on pages 74--76
of Bott's lectures~\cite{bott72} on characteristic classes and foliations.
Its original construction (in greater generality) was based on Bott's
vanishing theorem for Pontrjagin classes of normal bundles to integrable
subbundles and Haefliger's theory of classifying spaces for foliations,
cf.~\cite{pitt76}. The simple construction in terms of complex-valued
differential forms was inspired by the work of Godbillon and Vey.

Bott's lectures also contain the computation of the invariant
for a certain family of transversely holomorphic foliations on~$S^3$,
see Proposition~\ref{prop:GV-Poincare} below.
\end{rem}

By the comment after Definition~\ref{defn:tcc},
the Bott number is in particular an invariant
of taut contact circles. Observe that if the formally integrable
complex $1$-form $\omegac$ stems from a taut contact circle, then
so does the $1$-form $\rho\rme^{\rmi\theta}\omegac$ for any smooth, nowhere
zero real-valued function $\rho$ on $M$, and any constant angle $\theta$.
The corresponding contact circles are precisely those related to each other
by pointwise scaling and global rotation;
these form what in \cite{gego95,gego02} we called
the \emph{homothety class} of a contact circle. The computation
in (ii) shows that the Bott number is an invariant
of the homothety class.
\subsection{A generalised Gau{\ss}--Bonnet theorem}
In this section we discuss an instance where the Bott invariant
depends only on the $1$-di\-men\-sio\-nal foliation defined by
the transversely holomorphic flow, but not on the specific
transverse holomorphic structure. We shall deduce the Gau{\ss}--Bonnet
theorem for surfaces from this result.

\begin{thm}
\label{thm:GV-real}
Let $\omegac$ be a formally integrable complex $1$-form on $M$
for which there exists a pure imaginary $1$-form $\rmi\alpha$ such that
\[ \rmd\omegac=\rmi\alpha\wedge\omegac.\]
Then any other formally integrable complex $1$-form defining the same
$1$-dimensional foliation has the same Bott
invariant.
\end{thm}

\begin{rem}
\label{rem:Cartan}
The condition on the existence of the $1$-form $\rmi\alpha$ is equivalent
to
\[ \omega_1\wedge\rmd\omega_2=0=\omega_2\wedge\rmd\omega_1.\]
As a condition on $\omegac$ this can be written as
$\mathrm{Im}\,(\omegac\wedge \rmd\oomegac)=0$.
In the context of taut contact circles, this is what we called
a \emph{Cartan structure}, cf.\ \cite{gego95,gego02}.
\end{rem}

In general, the real and imaginary part of a formally integrable
complex $1$-form $\omegac$ define a transverse orientation
on the $1$-dimensional common kernel foliation. The complex
conjugate $\oomegac$ defines the opposite transverse orientation,
and the corresponding Bott invariants are complex
conjugates of each other.
In the situation of Theorem~\ref{thm:GV-real}, the
Bott invariant is a real number, so the choice of
coorientation is irrelevant.

\begin{proof}[Proof of Theorem~\ref{thm:GV-real}]
A simple pointwise calculation shows that, up to scaling by a
nowhere zero complex-valued function, any $1$-form defining the
same $1$-dimensional foliation and coorientation can be written as
\[ \omegac'=\omegac+\phi\oomegac\]
with some complex-valued function $\phi$ satisfying $|\phi|<1$. Then
\[ \rmd\omegac'=\rmi\alpha\wedge\omegac+(\rmd\phi-\rmi\phi\alpha)\wedge
\oomegac.\]
The requirement that $\omegac'$ be formally integrable gives
\begin{eqnarray*}
0 & = & \omegac'\wedge\rmd\omegac'\\ 
  & = & (\omegac+\phi\oomegac)\wedge(\rmi\alpha\wedge\omegac+
        (\rmd\phi-\rmi\phi\alpha)\wedge\oomegac)\\
  & = & (2\rmi \phi\alpha-\rmd\phi)\wedge\omegac\wedge\oomegac.
\end{eqnarray*}
This implies the existence of complex-valued functions $a,b$ such that
\[ 2\rmi\phi\alpha-\rmd\phi=a\omegac+b\oomegac.\]
Then $\rmd\omegac'$ can be rewritten as
\begin{eqnarray*}
\rmd\omegac' & = & \rmi\alpha\wedge (\omegac+\phi\oomegac)+
                   (\rmd\phi-2\rmi\phi\alpha)\wedge\oomegac\\
             & = & \rmi\alpha\wedge\omegac'-a\omegac\wedge\oomegac\\
             & = & (\rmi\alpha+a\oomegac)\wedge\omegac',
\end{eqnarray*}
which means that we may take
\[ \alphac'=\rmi\alpha+a\oomegac.\]
With this choice we have
\[ \rmd\oomegac=-\rmi\alpha\wedge\oomegac=-\alphac'\wedge\oomegac.\]
The argument in Section \ref{subsection:classical}~(i),
applied to the formally integrable $1$-form $\oomegac$, then shows
that the difference
\[ (\rmi\alpha)\wedge\rmd(\rmi\alpha)-\alphac'\wedge\rmd\alphac'\]
is exact.
\end{proof}

\begin{cor}[Gau{\ss}--Bonnet]
\label{cor:GB}
Let $\Sigma$ be a closed surface with a Riemannian metric of
Gau{\ss} curvature~$K$. The value of the integral $\int_\Sigma K\,\rmd A$
only depends on~$\Sigma$, not on the choice of metric.
\end{cor}

\begin{proof}
Let $\pi\co M\rightarrow\Sigma$ be the unit tangent bundle of~$\Sigma$.
Let us first assume that $\Sigma$ is orientable.
On $M$ we then have the standard Liouville--Cartan pair $\omega_1,\omega_2$,
cf.\ \cite[p.~149]{gego95}, \cite[Section~3]{gego02}, and
a connection $1$-form $\alpha$. These satisfy the structure equations
of a Cartan structure:
\begin{eqnarray*}
\rmd\omega_1 & = & \omega_2\wedge\alpha\\
\rmd\omega_2 & = & \alpha\wedge\omega_1\\
\rmd\alpha & = & (\pi^*K)\omega_1\wedge\omega_2.
\end{eqnarray*}
The complex $1$-form $\omegac:=\omega_1+\rmi\omega_2$ is then
formally integrable, with $\rmd\omegac=\rmi\alpha\wedge\omegac$.
When we change the metric or orientation on $\Sigma$, we can interpret this
as keeping the fibration $M\rightarrow\Sigma$, but
changing the transverse holomorphic structure on it.
By Theorem~\ref{thm:GV-real}, the total Gau{\ss} curvature
\[ \int_{\Sigma}K\,\rmd A=
\frac{1}{2\pi}\int_M(\pi^*K)\omega_1\wedge\omega_2\wedge\alpha=
\frac{1}{2\pi}\int_M\alpha\wedge\rmd\alpha\]
is, up to a factor $-1/2\pi$, the Bott invariant
determined solely by the fibration.

If $\Sigma$ is not orientable, we apply the preceding discussion
to an orientable double cover of~$\Sigma$.
\end{proof}
\section{Transversely holomorphic foliations on $S^3$}
\label{section:S3}
We now turn our attention to transversely holomorphic
foliations on the $3$-sphere~$S^3$. We shall introduce
two families of such foliations, and in Theorem~\ref{thm:all-P}
we show that this is a complete list. We also compute the
Bott invariant of these foliations.
\subsection{Poincar\'e foliations -- the parametric family}
In this section we study transversely holomorphic foliations
on $S^3$ induced from a formally integrable
complex $1$-form on $\C^2$ given by
\begin{equation}
\label{eqn:alpha-beta}
\omegac=\alpha z_1\,\rmd z_2-\beta z_2\,\rmd z_1
\end{equation}
for a pair $(\alpha,\beta)$ of complex numbers in the
so-called Poincar\'e domain.
A finite set of points in the complex plane is said to be in the
\emph{Poincar\'e domain}~\cite{arno88} if 
their convex hull does not contain the origin. For a pair
$(\alpha,\beta)$ this simply means that $\alpha,\beta\neq0$
and $\alpha/\beta\not\in\R^-$.

The reason for this restriction is provided by the following lemma,
which is implicit in~\cite{brun96}. In a wider context, this is
studied in~\cite{itsc05}.

\begin{lem}
\label{lem:P-domain}
The real and imaginary parts of $\omegac$ as in (\ref{eqn:alpha-beta})
induce pointwise linearly independent $1$-forms on $S^3\subset\C^2$,
and hence define a transversely holomorphic flow there,
if and only if $(\alpha,\beta)$ is in the Poincar\'e domain.
\end{lem}

\begin{proof}
Clearly both $\alpha$ and $\beta$ have to be non-zero,
otherwise the $1$-form $\omegac$ vanishes along
one of the Hopf circles $S^1\times\{0\}$ or
$\{0\}\times S^1\subset S^3\subset\C^2$.

Write $\omega_1,\omega_2$ for the real and imaginary part of
$\alpha z_1\,\rmd z_2-\beta z_2\,\rmd z_1$, respectively. The condition
for $\omega_1,\omega_2$ to induce pointwise linearly independent
$1$-forms on $S^3$ is that the plane field
$\DD:=\ker\omega_1\cap\ker\omega_2$
on $\C^2\setminus\{(0,0)\}$ be transverse to~$S^3$.

The plane field $\DD$ is in fact the complex line field spanned
by the holomorphic vector field $X:=\alpha z_1\partial_{z_1}+
\beta z_2\partial_{z_2}$. So we need to ensure that the real and
imaginary part of $X$ are not simultaneously tangent to~$S^3$.
This translates into
\[ 0\neq X(|z_1|^2+|z_2|^2)=\alpha|z_1|^2+\beta|z_2|^2\]
at all points $(z_1,z_2)\in S^3$, which is equivalent to $(\alpha,\beta)$
being in the Poincar\'e domain.
\end{proof}

By scaling $\omegac$ with a constant in $\C^*$, we may restrict
attention to Poincar\'e pairs of the form $(\alpha,\beta)=
(a,1-a)$ with $a\neq 0,1$ and $(1-a)/a\not\in\R^-$. This means
\[ a\in\PP:=(\C\setminus\R)\cup (0,1).\]

\begin{rem}
\label{rem:a-cc}
We claim that, as shown in \cite{gego95},
\[ \omega^a=\omega_1^a+\rmi\omega_2^a:=az_1\,\rmd z_2-(1-a)z_2\,\rmd z_1\]
defines a taut contact circle on $S^3$ if and only if
$0<\re(a)<1$, which describes a proper subset of~$\PP$.
Indeed, with $X:=az_1\partial_{z_1}+(1-a)z_2\partial_{z_2}$, and using
the fact that $\omega^a$ is formally integrable, one finds
\[ 2\omega_1^a\wedge\rmd\omega_1^a=\re(\oomega^a\wedge\rmd\omega^a)=
(X+\overline{X})\ip(\rmd\oz_1\wedge\rmd\oz_2\wedge\rmd z_1\wedge\rmd z_2).\]
So the contact circle condition is that $X+\overline{X}$ be transverse
to~$S^3$. From
\[ (X+\overline{X})(|z_1|^2+|z_2|^2)=2\re(a)|z_1|^2+2(1-\re(a))|z_2|^2\]
the claim follows.

Even for a general $a\in\PP$, the pair $(\omega_1^a,\omega_2^a)$
will satisfy the contact circle condition near at least one
of the Hopf circles, since $\re(a)$ and $1-\re(a)$ never vanish
simultaneously. This observation will be relevant in the
proof of Theorem~\ref{thm:unique-hol}.
\end{rem}

\begin{defn}
The $1$-dimensional foliations $\FF^a$ on $S^3$ defined
by the $\omega^a$ with $a\in\PP$ are
said to constitute the \emph{parametric family} of
\emph{Poincar\'e foliations}.
\end{defn}

We shall say more about this terminology in
Section~\ref{subsection:Poincare-discrete}. The symbol $\FF^a$ is meant
to denote an oriented and cooriented foliation: the coorientation is
the one defined by~$\omega^a$, the orientation is the one which
together with this coorientation gives the standard orientation of~$S^3$.
No specific transverse holomorphic structure is meant to be
implied by the symbol~$\FF^a$. One of our main objectives will be
to investigate to what extent the foliation $\FF^a$ alone determines
the transverse holomorphic structure or the homothety class
of the contact circle.

The map $(z_1,z_2)\mapsto (-z_2,z_1)$ defines an
orientation-preserving diffeomorphism of $S^3$ and pulls back
$\omega^a$ to $\omega^{1-a}$. So $(\FF^a,\omega^a)$ and
$(\FF^{1-a},\omega^{1-a})$ are diffeomorphic as transversely holomorphic
foliations. The set
\[ \MM:=\{ a\in\C\co 0<\re(a)<1\}/(a\sim 1-a) \]
constitutes the non-discrete part of the moduli space of taut contact
circles on~$S^3$, see \cite{gego95,gego02}.

The existence of a diffeomorphism between the transversely holomorphic
flows defined by $\omega^a$ and $\omega^{1-a}$ is reflected in the
following computation of their Bott invariant.

\begin{prop}
\label{prop:GV-Poincare}
For $a\in\PP$, the Bott invariant of $\omega^a$ equals
\[ \frac{-4\pi^2}{a(1-a)}.\]
\end{prop}

\begin{proof}
On $\C^2\setminus\{(0,0)\}$ we have $\rmd\omega^a=\alpha^a\wedge\omega^a$
with
\[ \alpha^a:=\frac{1}{|z_1|^2+|z_2|^2}\Bigl(
\frac{1}{a}\oz_1\,\rmd z_1+\frac{1}{1-a}\oz_2\,\rmd z_2\Bigr).\]
On $TS^3$ we compute
\begin{eqnarray*}
\alpha^a\wedge\rmd\alpha^a
 & = & \frac{1}{a(1-a)}(\oz_1\,\rmd z_1\wedge\rmd\oz_2\wedge\rmd z_2+
                        \oz_2\,\rmd z_2\wedge\rmd\oz_1\wedge\rmd z_1)\\
 & = & \frac{-4}{a(1-a)}(\oz_1\partial_{\oz_1}+\oz_2\partial_{\oz_2})\ip
       (\rmd x_1\wedge\rmd y_1\wedge\rmd x_2\wedge\rmd y_2).
\end{eqnarray*}
The real part of $\oz_1\partial_{\oz_1}$ equals $(x_1\partial_{x_1}+y_1
\partial_{y_1})/2$; the imaginary part $(x_1\partial_{y_1}-
y_1\partial_{x_1})/2$ is tangent to $S^3$. It follows that
on $TS^3$ we have
\[ \alpha^a\wedge\rmd\alpha^a
  =  \frac{-2}{a(1-a)}\sum_{j=1}^2(x_j\partial_{x_j}+y_j\partial_{y_j})
       \ip (\rmd x_1\wedge\rmd y_1\wedge\rmd x_2\wedge\rmd y_2),\]
which integrates to
\[ \frac{-2}{a(1-a)}\mathrm{vol}\,(S^3)=\frac{-4\pi^2}{a(1-a)}.\qedhere\]
\end{proof}

\begin{rem}
In \cite{asuk10}, the $1$-form $\alphac$ in the construction of
the Bott invariant is defined via the equation $\rmd\omega= 2\pi\rmi\alphac
\wedge\omega$. With this normalising factor $2\pi\rmi$, the Bott
invariant of $\omega^a$ takes the value $1/a(1-a)$. The definition
without this factor, which is also the one in \cite[p.~8]{pitt76},
is notationally more convenient for the computations in
Section~\ref{section:thf-tcc}.
\end{rem}

The map
\[ \begin{array}{ccc}
\PP & \longrightarrow & \C\setminus\R_0^-\\
a   & \longmapsto     & a(1-a)
\end{array} \]
is a double branched covering, branched at the point $a=1/2$.
This can best be seen by writing $a=\frac{1}{2}+b$; then $a(1-a)=\frac{1}{4}
-b^2$. This map descends to a bijection
\[ \begin{array}{ccc}
\PP/(a\sim 1-a) & \longrightarrow & \C\setminus\R_0^-\\
\mbox{$[a]$}    & \longmapsto     & a(1-a).
\end{array} \]
Hence, with Proposition~\ref{prop:GV-Poincare} we deduce:

\begin{cor}
\label{cor:P-fol}
Up to orientation-preserving diffeomorphism, a Poincar\'e foliation
$\FF^a$ with the transverse holomorphic structure given by $\omega^a$ is
determined, within the class of all pairs $(\FF^a,\omega^a)$,
by its Bott invariant.\qed
\end{cor}

This means that we may regard $\C\setminus\R_0^-$ as the
moduli space of Poincar\'e foliations $(\FF^a,\omega^a)$
in the parametric family.
In particular, the image of $\MM$ under the map $[a]\mapsto a(1-a)$,
which is the convex open set
$\{x+\rmi y\in\C\co x>y^2\}$, can be thought of as (one component of)
the moduli space of taut contact circles on~$S^3$, see~\cite{gego02}.

\begin{rem}
\label{rem:alpha-a}
The $\alpha^a$ used in the proof of Proposition~\ref{prop:GV-Poincare}
is the most convenient one for computing the Bott invariant.
However, it may be replaced by
\[ \frac{1}{a|z_1|^2+(1-a)|z_2|^2}(\oz_1\,\rmd z_1+\oz_2\,\rmd z_2).\]
For $a\in(0,1)$, that is, for $a$ in the real part of~$\PP$
(or~$\MM$), the restriction of this $1$-form to $TS^3$ is pure
imaginary, since
\[ \oz_1\,\rmd z_1+\oz_2\,\rmd z_2+z_1\,\rmd\oz_1+z_2\,\rmd\oz_2=
2(x_1\,\rmd x_1+y_1\,\rmd y_1+x_2\,\rmd x_2+y_2\,\rmd y_2).\]
So for these $\omega^a$ Theorem~\ref{thm:GV-real} applies.
Alternatively, one may check that
\[ \mathrm{Im}\,(\omega^a\wedge\rmd
\overline{\omega^a})=0\;\;\text{for}\;\; a\in (0,1).\]
\end{rem}
\subsection{Poincar\'e foliations -- the discrete family}
\label{subsection:Poincare-discrete}
In \cite{gego95} it was shown that the moduli space of 
homothety classes of taut contact circles on $S^3$
is given by the disjoint union of $\MM$ and the countable family
defined by
\[ \omega_n:=nz_1\,\rmd z_2-z_2\,\rmd z_1+z_2^n\,\rmd z_2,\;\;\;
n\in\N:=\{1,2,3,\ldots\}.\]
Write $\FF_n$ for the oriented and cooriented
$1$-dimensional foliation on $S^3$ defined
by~$\omega_n$.

\begin{defn}
We say the $\FF_n$, $n\in\N$, make up the \emph{discrete family}
of \emph{Poincar\'e foliations} on~$S^3$.
\end{defn}

A larger part of the following theorem is due to
Brunella~\cite{brun96} and Ghys~\cite{ghys96}, but they do not describe
the explicit models. A list of these models is also contained
in~\cite[Theorem~2.1]{itsc05}.

\begin{thm}
\label{thm:all-P}
The $\FF^a$, $a\in\PP$, and the $\FF_n$, $n\in\N$, exhaust all
foliations on $S^3$ admitting a transverse holomorphic structure.
\end{thm}

\begin{proof}
According to \cite{brun96,ghys96}, any foliation on $S^3$
admitting a transverse holomorphic structure is a
\emph{Poincar\'e foliation}, i.e.\ it is a foliation --- on
a small sphere around the origin $(0,0)\in\C^2$ ---
induced by a holomorphic vector field with a singularity at $(0,0)$
whose linearisation at the origin has a pair of eigenvalues in the
Poincar\'e domain.

According to the Poincar\'e--Dulac theorem \cite[p.~190]{arno88},
\cite{cmv05}, such a singularity is biholomorphic to
a polynomial normal form, where the non-linear terms
come from resonances. For a singularity in $\C^k$ this means the following.
Write $\blambda=(\lambda_1,\ldots,\lambda_k)$ for the eigenvalues
of the linearisation. A resonance is a multi-index $\bm=(m_1,\ldots, m_k)
\in\N_0^k$
of non-negative integers with $m_1+\cdots+m_k\geq 2$, for which there
is a $j\in\{1,\ldots,k\}$ such that
\[ \langle\blambda,\bm\rangle-\lambda_j=0.\]
Any such resonance then gives rise to a monomial term
$c_{\bm,j}z_1^{m_1}\cdots z_k^{m_k}\,\partial_{z_j}$ in the
polynomial normal form.

In complex dimension two, by rescaling we may assume that
$\lambda_1=a$ and $\lambda_2=1-a$, with $a\in\PP$. The resonance
condition for $\lambda_1$ then becomes
\[ am_1+(1-a)m_2=a.\]
With $m_1,m_2\in\N_0$ this implies $a\in\PP\cap\R=(0,1)$, and further
$m_1=0$ and $m_2=a/(1-a)$. So the resonance condition
is $n:=a/(1-a)\in\{2,3,\ldots\}$.
The resonance condition for $\lambda_2$ leads to
$(1-a)/a\in\{2,3,\ldots\}$, which we can ignore by symmetry.

So the only resonant term is
\[ z_2^n\partial_{z_1}\;\;\text{for}\;\;\frac{a}{1-a}=n\in\{2,3,\ldots\}.\]
This condition on $a$ rules out the case of a double eigenvalue $a=1/2$
in the linearised singularity, so the corresponding normal form is
\[ \frac{n}{1+n} (z_1+cz_2^n)\,\rmd z_2-\frac{1}{1+n}z_2\,\rmd z_1.\]
By rescaling and pull-back under the map $(z_1,z_2)\mapsto (cnz_1,z_2)$
for $c\neq 0$, we obtain the $\omega_n$, $n\geq 2$, introduced above.

In the non-resonant case, we obtain $\omega^a$, $a\in\PP$, if the
linearisation is diagonalisable, and $\omega_1$ if it is not.
\end{proof}

Our computation of the Bott invariant of the $\omega_n$
depends crucially on the moduli theory of taut contact circles.

\begin{prop}
The Bott invariant of $\omega_n$ equals
\[ -4\pi^2\,\frac{(n+1)^2}{n}.\]
\end{prop}

\begin{proof}
In \cite[\S 6]{gego95} we discussed the following `jump'
homotopy, which mirrors a phenomenon in the moduli theory
of Hopf surfaces discovered by Kodaira and Spencer~\cite{kosp58}.
For given $n\in\N$, consider the family
\[ \omega_n^{\lambda}:=nz_1\,\rmd z_2-z_2\,\rmd z_1+\lambda
z_2^n\,\rmd z_2,\;\;\; \lambda\in [0,1].\]
For $\lambda\in(0,1]$ these complex $1$-forms all define the same taut contact
circle, up to homothety and diffeomorphism. For $\lambda=0$
we obtain the taut contact circle homothetic to
\[ \omega^{n/(n+1)}=\frac{n}{n+1}\, z_1\,\rmd z_2-\frac{1}{n+1}\, z_2\,
\rmd z_1\]
from the parametric family.

Although the equivalence class of the taut contact circle jumps at
$\lambda=0$, the Bott invariant will depend continuously
on $\lambda$ for all $\lambda\in [0,1]$ and hence,
being constant on $(0,1]$, will
be identically equal to that of $\omega^{n/(n+1)}$.
\end{proof}
\section{Logarithmic monodromy}
\label{section:log}
In order to describe the geometry of a transversely holomorphic
foliation, we study the logarithmic monodromy
along a closed leaf. It is best to explain the concept
in a concrete case.

Thus, consider a Poincar\'e foliation $\FF^a$, with
transversely holomorphic structure given by $\omega^a$,
and with the corresponding orientation of the leaves.
For any $a\in\PP$, the two
Hopf circles $S^1\times\{0\}$ and $\{0\}\times S^1$ constitute
closed leaves of~$\FF^a$.

Either of these Hopf circles, just like any other knot in~$S^3$, comes with
a preferred trivialisation (up to homotopy) of its normal bundle, namely, the
surface framing defined by a Seifert surface of the knot.
The transverse holomorphic structure $J$ then determines an oriented
conformal framing: take any vector field $Z$ along the knot
which is tangent to the Seifert surface, and declare that the
rotate of $Z$ through an angle $\pi/2$ be equal to $JZ$.
For the Hopf circle $S^1\times\{0\}$, such a Seifert surface
is given by the disc
\[ \{ (r\rme^{\rmi\theta},\sqrt{1-r^2})\co r\in [0,1],\; \theta\in\R\}
\subset S^3.\]
This corresponds to the oriented conformal framing given by the oriented
basis $(\partial_{x_2},\partial_{y_2})$ of tangent vector
fields along the Hopf circle, or by the type $(1,0)$ complex
tangent vector~$\partial_{z_2}$.

Such a framing allows us to identify a neighbourhood of an
oriented closed leaf $\gamma$
with a neighbourhood of $S^1\times\{0\}$ in $S^1\times\C$.
Fix the transversal $\{1\}\times\C$ to $S^1\times\{0\}$.
The oriented foliation then determines a family of germs
of holomorphic maps $\varphi_t\co(\C,0)\rightarrow (\C,0)$
by writing the intersection point of the leaf through
$(1,z)$ with the transversal $\{\rme^{\rmi t}\}\times\C$ as
$(\rme^{\rmi t},\varphi_t(z))$.

We can then make a continuous choice of logarithm $\log\varphi_t'(0)$
with $\log\varphi_0'(0)=\log 1 =0$.
A different identification of $\gamma$ with $S^1$ and a homotopy
of the framing will change
the map $\varphi_t$ by conjugation and homotopy rel~$\{0,2\pi\}$,
so the following quantity associated with a closed
leaf is independent of choices.

\begin{defn}
\label{defn:log-m}
The \emph{logarithmic monodromy} of the closed leaf $\gamma$ is
$\log\varphi_{2\pi}'(0)$.
\end{defn}

\begin{rem}
Our notion of logarithmic monodromy may be interpreted as a simple
instance of the residue theory developed by Heitsch~\cite{heit78},
see also \cite[Chapter~5]{asuk10} and \cite[Example~6.1]{asuk04},
for instance.
\end{rem}

Notice that although we need
a transverse holomorphic structure to define the logarithmic monodromy
of a closed leaf,
the value of this monodromy is completely determined by
the oriented and cooriented foliation:

\begin{lem}
The logarithmic monodromy is independent of the choice of
transverse holomorphic structure inducing a given
transverse orientation.
\end{lem}

\begin{proof}
Let one transverse holomorphic structure be given by the
formally integrable $1$-form $\omegac$. Then, as in the proof
of Theorem~\ref{thm:GV-real}, we observe that any other
$1$-form defining the same cooriented foliation can be scaled to
\[ \omegac'=\omegac+\phi\oomegac\]
with $|\phi|<1$. If we choose an $\alphac$ such that
$\rmd\omegac=\alphac\wedge\omegac$, the condition for $\omegac'$ to be
formally integrable becomes
\[ (\phi\alphac-\phi\oalphac-\rmd\phi)\wedge\omegac\wedge\oomegac=0.\]
This condition is linear in~$\phi$, so it
follows that $\omegac+\lambda\phi\oomegac$, $\lambda\in [0,1]$,
defines a homotopy of transverse holomorphic structures.

Thus, changing the transverse holomorphic structure once again amounts
to changing the map $\varphi_t$ by conjugation and homotopy rel~$\{0,2\pi\}$.
\end{proof}

If we change the orientation of the foliation, the logarithmic
monodromy changes its sign; changing the coorientation amounts to
taking the complex conjugate of the logarithmic monodromy.

\begin{prop}
\label{prop:log-a}
For $a\in\PP$, the logarithmic monodromy of
$S^1\times\{0\}$ in $\FF^a$ is $2\pi\rmi (1-a)/a$,
that of $\{0\}\times S^1$ is $2\pi\rmi a/(1-a)$.
\end{prop}

\begin{proof}
By the proof of Lemma~\ref{lem:P-domain}, the complex $1$-form
$\omega^a$ defines a plane field on $\C^2\setminus\{(0,0)\}$
transverse to~$S^3$. Therefore, for the computation of
the logarithmic monodromy of $S^1\times\{0\}$ we may replace
$S^3$ by $S^1\times\C$, which has the same tangent
spaces along that Hopf circle. Moreover, the trivialisation
$S^1\times\C$ of the normal bundle accords with the transverse
holomorphic structure and trivialisation defined by~$\partial_{z_2}$.

The complex $1$-form induced by $\omega^a$ on $S^1\times\C$ can be
written as
\[ a\rme^{\rmi\theta}\,\rmd z-(1-a)\rmi\rme^{\rmi\theta}z\,\rmd\theta.\]
So the induced flow is given by the vector field
\[ \partial_{\theta}+\frac{1-a}{a}\rmi z\partial_z,\]
and the flow lines are parametrised by
\[ t\longmapsto\bigl(\rme^{\rmi t},z\rme^{\rmi t(1-a)/a}\bigr).\]
The claimed logarithmic monodromy follows. For $\{0\}\times S^1$
the computation is analogous.
\end{proof}

\begin{ex}
The orientation-preserving diffeomorphism of $S^3$ given by
$(z_1,z_2)\mapsto(\oz_1,\oz_2)$ pulls back $\overline{\omega^a}$ to
$\omega^{\overline{a}}$. So this diffeomorphism sends
$\FF^{\overline{a}}$ to $\FF^a$ with reversed orientation and coorientation,
and it maps each Hopf circle to itself. This is consistent with
the computation in the preceding proposition, since the negative
complex conjugate of $2\pi\rmi(1-a)/a$ is $2\pi\rmi(1-\oa)/\oa$.
\end{ex}

We now turn to the discrete family.
The Hopf circle $S^1\times\{0\}$ is a closed leaf
of each of the foliations~$\FF_n$.

\begin{prop}
\label{prop:log-n}
For $n\in\N$, the logarithmic monodromy of $S^1\times\{0\}$ in
$\FF_n$ equals $2\pi\rmi/n$.
\end{prop}

\begin{proof}
As in the preceding proof, we replace $S^3$ by $S^1\times\C$,
where the complex $1$-form induced by $\omega_n$ can be written as
\[ n\rme^{\rmi\theta}\,\rmd z-\rmi\rme^{\rmi\theta}z\,\rmd\theta+
z^n\,\rmd z.\]
The common kernel flow near $S^1\times\{0\}$
is given by the vector field
\[ \partial_{\theta}+\frac{\rmi z}{n+\rme^{-\rmi\theta}z^n}\,\partial_z
=\partial_{\theta}+\frac{\rmi z}{n}\,\partial_z+O(z^2).\]
It follows that the logarithmic monodromy is the same as for the
flow
\[ t\longmapsto\bigl(\rme^{\rmi t},z\rme^{\rmi t/n}\bigr).\qedhere\]
\end{proof}
\section{Topology of the flows}
\label{section:topology}
In this section we give explicit descriptions of the
Poincar\'e foliations.
Specifically, we determine the closed leaves and the limiting behaviour
of the non-closed ones.
\subsection{The parametric family}
As observed earlier, each $\FF^a$ contains the Hopf circles
$S^1\times\{0\}$ and $\{0\}\times S^1$ as closed leaves. Depending
on the value $a\in\PP$, these may be the only closed leaves, or all
leaves may be closed:

\begin{prop}
\label{prop:P-leaves}
For $a\in\C\setminus\R$, the Hopf circles are the only closed
leaves of $\FF^a$. Every other leaf is asymptotic to
the two Hopf circles, one at either end.

For $a\in(0,1)$, all leaves apart from the Hopf circles are curves
of constant slope $a/(1-a)$ on the
Hopf tori $\{|z_1|=\mathrm{const.}\}$, regarded as boundary of
a tubular neighbourhood of the Hopf circle~$S^1\times\{0\}$.
\end{prop}

\begin{proof}
In the complement of the Hopf link we can write
\[ \omega^a=az_1z_2\,\rmd\Bigl(\log z_2-\frac{1-a}{a}\log z_1\Bigr).\]
So each leaf of $\FF^a$ in this domain can be
described by an equation
\[ \log z_2-\frac{1-a}{a}\log z_1=l_0+i\theta_0\]
for some real constants $l_0,\theta_0$.

Write $z_j=r_j\rme^{\rmi\theta_j}$, $j=1,2$, and use
$r_1\in(0,1),\theta_1,\theta_2$ as coordinates outside the Hopf link.
Define $u,v\in\R$ by $u+\rmi v=(1-a)/a$.
The leaves are then given by equations as follows:
\begin{equation}
\label{eqn:leaves}
\left\{\begin{array}{rcl}
\displaystyle{\log\sqrt{1-r_1^2} -
  u\log r_1 +v\theta_1}             & = & l_0,\\[2mm]
\theta_2-u\theta_1-v\log r_1        & = & \theta_0.
\end{array}\right.
\end{equation}
Notice that the ambiguity in the definition of the complex logarithm
is absorbed into the constants.

For $a\in\C\setminus\R$, and hence $v\neq 0$, these equations allow us
to express $\theta_1,\theta_2$ as functions of~$r_1\in(0,1)$, and so
they describe leaves asymptotic to the two Hopf circles:
\begin{equation}
\label{eqn:asymptotic}
\left\{\begin{array}{rcl}
\theta_1 & = & \displaystyle{\frac{1}{v}\Bigl(l_0+u\log r_1-
               \log\sqrt{1-r_1^2}\,\Bigr)},\\[3mm]
\theta_2 & = & \displaystyle{\theta_0+\frac{u}{v}\, l_0+\frac{1}{v}
               \Bigl((u^2+v^2)\log r_1-u\log\sqrt{1-r_1^2}\,\Bigr)}.
\end{array}\right.
\end{equation}
The precise asymptotic behaviour in dependence on the
value of $a\in\C\setminus\R$ will be discussed below.

For $a\in (0,1)$, so that $v=0$ and $u=(1-a)/a$,
equations (\ref{eqn:leaves}) can be written as
\begin{equation}
\label{eqn:slope}
\left\{\begin{array}{rcl}
\displaystyle{\log\sqrt{1-r_1^2}-\frac{1-a}{a}\log r_1}   & = & l_0,\\[3mm]
\displaystyle{\theta_2-\frac{1-a}{a}\,\theta_1}           & = & \theta_0.
\end{array}\right.
\end{equation}
The first of these equations describes a Hopf torus
$\{r_1=\mathrm{const.}\}$. (It is straightforward to check
that for each $a\in (0,1)$ the left-hand side of the first equation
defines a strictly monotone decreasing function in $r_1$
with image all of~$\R$.)
The second equation defines a curve of constant slope
$a/(1-a)$ on that torus. The foliation, including the Hopf link,
can be described as the flow of the Killing vector field
$a\partial_{\theta_1}+(1-a)\partial_{\theta_2}$ for the standard metric
on~$S^3$.
\end{proof}

The preceding proposition tells us that the leaves of $\FF^a$ are all closed
if and only if $a\in(0,1)\cap\Q$. If $a=1/2$, the foliation defines
the Hopf fibration of~$S^3$. For other rational values of~$a$,
the foliation defines a Seifert fibration with at least one
singular fibre.

\begin{prop}
\label{prop:Seifert}
Given $a\in (0,1)\cap\Q$,
write $a/(1-a)=p_1/p_2$ with $p_1,p_2$
coprime natural numbers. Choose integers $q_1',q_2'$ such that
\[ \begin{vmatrix}
p_1   & p_2\\
-q_1' & q_2'
\end{vmatrix}
=1,\]
and define integers $m_1,m_2$ by the requirement that $q_j'=m_jp_j+q_j$ with
$0\leq q_j<p_j$, $j=1,2$.

Then the foliation $\FF^a$ defines a Seifert fibration
of $S^3$ with unnormalised Seifert invariants
\[ \bigl(g=0,(p_1,q_1'),(p_2,q_2')\bigr)\]
and normalised Seifert invariants
\[ \bigl(g=0,b=m_1+m_2,(p_1,q_1),(p_2,q_2)\bigr).\]
The quotient orbifold is $S^2(p_1,p_2)$.
\end{prop}

\begin{proof}
We follow the recipe in \cite{nera78} for computing the Seifert invariants;
for easy reference we retain their notation.
By equation (\ref{eqn:slope}) in
the preceding proof, the leaves of $\FF^a$ are the orbits of the
$S^1$-action on $S^3$ given by
\[ \theta(z_1,z_2)=(\rme^{\rmi p_1\theta}z_1,\rme^{\rmi p_2\theta}z_2).\]
The singular orbits are the Hopf circles
$O_1=S^1\times\{0\}$ and $O_2=\{0\}\times S^1$. Disjoint invariant tubular
neighbourhoods of these two orbits are given by
\[ T_1=\{|z_1|^2\geq 3/4\}\;\;\text{and}\;\; T_2=\{|z_2|^2\geq 3/4\}.\]
Set
\[ M_0=S^3\setminus\Int(T_1\cup T_2).\]
Then $M_0\rightarrow M_0/S^1$ is an $S^1$-bundle over an annulus,
and the quotient orbifold $S^3/S^1$ is a $2$-sphere with two cone
points of order $p_1,p_2$, respectively, given by the multiplicity
of the singular orbits.

Write $\mu_j$ for the meridian of $T_j$. We think of these two
curves as a homology class of curves on any Hopf torus.
Take $\lambda_1:=\mu_2$ and $\lambda_2:=\mu_1$ as the
standard longitudes.
The non-singular orbits are in the class $p_1\lambda_1+p_2\lambda_2$.
A homologically dual curve is $q_1'\lambda_1-q_2'\lambda_2$.
This defines a section
$R\subset M_0$ of the $S^1$-bundle $M_0\rightarrow M_0/S^1$.
Notice that the homological intersection of these two
curves on $\partial T_1$ is
\[ (p_2\mu_1+p_1\lambda_1)\bullet(-q_2'\mu_1+q_1'\lambda_1)=1.\]
It follows that the orientation of $R$ compatible with the
standard orientation of $S^3$ and the orientation of the
$S^1$-orbits is the one for which the oriented boundary curves of this
section are
\[ R_1:=q_1'\lambda_1-q_2'\lambda_2\subset\partial T_1\]
and
\[ R_2:=-(q_1'\lambda_1-q_2'\lambda_2)\subset\partial T_2.\]
In the respective solid torus these curves are homologous to
\[ q_1'O_1\subset T_1\;\;\text{and}\;\; q_2'O_2\subset T_2.\]
This yields the unnormalised Seifert invariants.
The normalised Seifert invariants follow from the equivalences
described in \cite[Theorem~1.1]{nera78}.
\end{proof}

\begin{rem}
For $p_1=p_2=1$, the quotient orbifold $S^2(p_1,p_2)$ is simply the
$2$-sphere. If exactly one of the $p_i$ equals~$1$,
we have a tear-drop. If both $p_1$ and $p_2$ are greater
than~$1$, the orbifold is a spindle. Thus, all possible
tear-drops are realised, but only spindles with coprime
multiplicities at the cone points.
\end{rem}

We now take a closer look at the asymptotic behaviour of the
leaves of $\FF^a$ for $a\in\C\setminus\R$, described by
equations~(\ref{eqn:asymptotic}).
Recall that $u+\rmi v=(1-a)/a$. If we write $a=x+\rmi y$, this gives
\[ u=\frac{x-(x^2+y^2)}{x^2+y^2}.\]
So the case $u=0$ is equivalent to the condition $x=x^2+y^2$, which is the
same as $|a-\frac{1}{2}|=\frac{1}{2}$. Similarly, we have
\[ u>0\;\;\text{if and only if}\;\; |a-\frac{1}{2}|<\frac{1}{2}\]
and
\[ u<0\;\;\text{if and only if}\;\; |a-\frac{1}{2}|>\frac{1}{2}.\]
The imaginary part of $(1-a)/a$ is
\[ v= -\frac{y}{x^2+y^2},\]
which is always non-zero for $a\in\C\setminus\R$.

We write $\FF^a_1,\FF^a_2$ for the $2$-dimensional foliations on
the complement of the Hopf link defined by only the first or the
second equation in~(\ref{eqn:asymptotic}), respectively.
Then $\FF^a=\FF^a_1\cap\FF^a_2$.

\vspace{1mm}

\noindent
{\sc First case:} $u=0$. Here the limiting behaviour of $\theta_1,\theta_2$
is described by
\[ \left.\begin{array}{rcl}
\theta_1 & \longrightarrow & l_0/v\\
\theta_2 & \longrightarrow & -\sign (v)\,\infty
\end{array}\right\}
\;\;\text{for $r_1\searrow 0$} \]
and
\[ \left.\begin{array}{rcl}
\theta_1 & \longrightarrow & \sign(v)\,\infty\\
\theta_2 & \longrightarrow & \theta_0
\end{array}\right\}
\;\;\text{for $r_1\nearrow 1$}. \]
So the leaves of $\FF^a$ approach a limiting angle in the
direction transverse to the respective Hopf circle, and
they circle infinitely often in the direction
parallel to that Hopf circle.

The leaves of $\FF^a_1$ are open cylinders asymptotic at one end
to the Hopf circle $\{ r_1=0\}=O_2$, with a well-defined tangent plane
determined by the limiting angle~$\theta_1$.
Thus, near $O_2$ the foliation $\FF^a_1$ looks like an open
book near its binding. At the other end,
the cylinder sits like an ever thinner tube
around the Hopf circle $\{r_1=1\}=O_1$, winding infinitely often along
it.

\vspace{1mm}

\noindent
{\sc Second case:} $u>0$. Here $\theta_1$ and $\theta_2$ are monotone
functions of $r_1$ with
\[ \left.\begin{array}{rcl}
\theta_1 & \longrightarrow & -\sign (v)\,\infty\\
\theta_2 & \longrightarrow & -\sign (v)\,\infty
\end{array}\right\}
\;\;\text{for $r_1\searrow 0$} \]
and
\[ \left.\begin{array}{rcl}
\theta_1 & \longrightarrow & \sign(v)\,\infty\\
\theta_2 & \longrightarrow & \sign(v)\,\infty
\end{array}\right\}
\;\;\text{for $r_1\nearrow 1$}. \]

The cylindrical leaves of $\FF^a_1$ tube towards $O_1$ as before, but now the
other end of each cylinder scrolls towards $O_2$, encircling it infinitely
often.

\vspace{1mm}

\noindent
{\sc Third case:} $u<0$. In this case we have
\[ \left.\begin{array}{rcr}
\theta_1 & \longrightarrow & \sign (v)\,\infty\\
\theta_2 & \longrightarrow & -\sign (v)\,\infty
\end{array}\right\}
\;\;\text{for $r_1\searrow 0$} \]
and
\[ \left.\begin{array}{rcr}
\theta_1 & \longrightarrow & \sign(v)\,\infty\\
\theta_2 & \longrightarrow & -\sign(v)\,\infty
\end{array}\right\}
\;\;\text{for $r_1\nearrow 1$}. \]
One checks easily that the derivatives of $\theta_1$ and $\theta_2$
with respect to $r_1$ both change sign exactly once.
The cylindrical leaves of $\FF^a_1$ tube towards $O_1$ and scroll
towards $O_2$ as in the second case, but now they change the
$\theta_1$-direction once, making them look like sombreros,
see Figure~\ref{figure:sombrero}.

\vspace{1mm}

In all three cases, the cylindrical leaves of $\FF^a_2$ show
the analogous behaviour, with the roles of the two Hopf circles
interchanged.

\begin{figure}[h]
\labellist
\small\hair 2pt
\pinlabel $\theta_1$ [l] at 218 503
\pinlabel $O_1$ [l] at 210 13
\pinlabel $O_2$ [bl] at 378 205
\endlabellist
\centering
\includegraphics[scale=0.6]{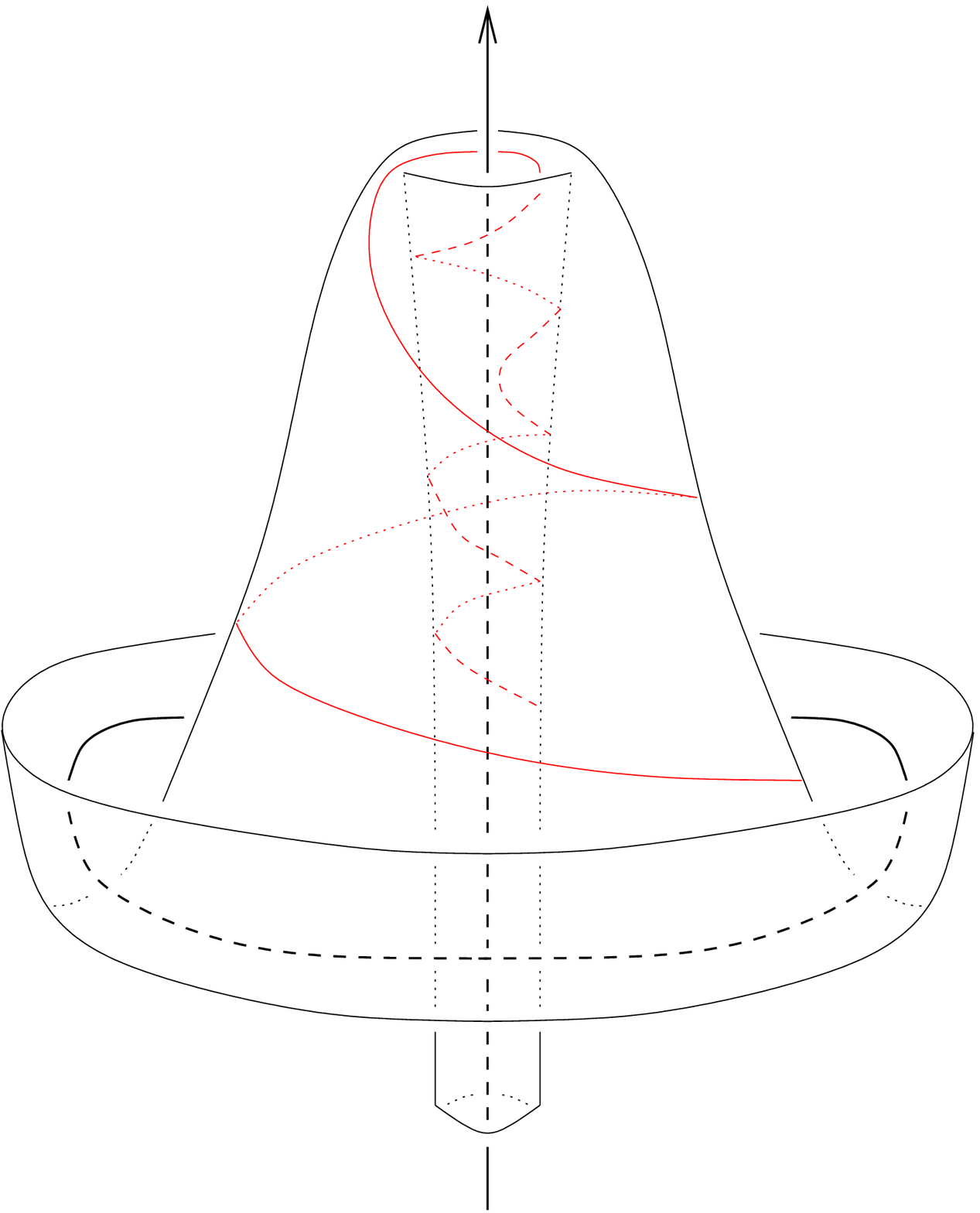}
  \caption{The `sombrero'.}
  \label{figure:sombrero}
\end{figure}
\subsection{The discrete family}
\label{subsection:topology-discrete}
For each $n\in\N$, the $1$-form $\omega_n$ defined in
Section~\ref{subsection:Poincare-discrete} may be regarded as
a holomorphic $1$-form on $\C^2$. Outside the origin, it defines
a foliation $\CC_n:=\ker\omega_n$ by holomorphic curves,
which we refer to as \emph{complex leaves}.

The complex line $\C\times\{0\}$ is a leaf of $\CC_n$, and it
intersects $S^3$ in the closed leaf $S^1\times\{0\}$ of~$\FF_n$.
On the complement $\C\times\C^*$ of that complex line,
the $1$-form $\omega_n$ can be written as
\[ \omega_n=z_2^{n+1}\,\rmd\Bigl(\log z_2-\frac{z_1}{z_2^n}\Bigr).\]
From this description, which we shall use to analyse the topology of $\FF_n$
in the complement
\[ S^3_0:=S^3\setminus\bigl(S^1\times\{0\}\bigr)\]
of the Hopf circle $S^1\times\{0\}$, we see
that each leaf of $\CC_n$ in $\C\times\C^*$ is given by an equation
\begin{equation}
\label{eqn:n-complex-leaf}
\log z_2-\frac{z_1}{z_2^n}=c_0
\end{equation}
for some complex constant~$c_0$. Observe that the solution set
of this equation is the image of the injective map
\[ \C\ni w\longmapsto \bigl((w-c_0)\rme^{nw},\rme^w\bigr),\]
so it is indeed connected. We shall see that the intersection
of each complex leaf with $S^3$ is also connected, and thus constitutes
a leaf of~$\FF_n$.

\begin{prop}
\label{prop:n-asymptotic}
For each $n\in\N$, the Hopf circle $S^1\times\{0\}$ is the only
closed leaf of~$\FF_n$. Every other leaf is asymptotic to this
Hopf circle at both ends.
\end{prop}

\begin{proof}
We take $n\in\N$ as given and suppress it from the notation
whenever appropriate. Let $\hat{\C}=\C\cup\{\infty\}$ be the
Riemann sphere, and consider the Seifert fibration
\[ \begin{array}{rccc}
\pi_n\co & S^3       & \longrightarrow & \hat{\C}\\[1mm]
         & (z_1,z_2) & \longmapsto     & \displaystyle{\frac{z_1}{z_2^n}},
\end{array} \]
with fibres given by the orbits of the $S^1$-action
\[ \theta(z_1,z_2)=(\rme^{\rmi n\theta}z_1,\rme^{\rmi\theta}z_2).\]

On $\C\subset\hat{\C}$ we use the coordinate $z=x+\rmi y$.
As before we write $z_j=r_j\rme^{\rmi\theta_j}$.
Since $z_1=(x+\rmi y)z_2^n$ for $z_2\neq 0$,  on
$S^3_0=S^3\setminus\pi_n^{-1}(\infty)$
the radius $r_2$ is defined implicitly as a smooth function $r_2(x,y)$
(depending on~$n$) by the equation
\begin{equation}
\label{eqn:r2}
(x^2+y^2)r_2^{2n}+r_2^2=1,\;\;\; r_2>0.
\end{equation}
Thus, $S^3_0$ can be parametrised in terms of $x,y,\theta_2$ by
\[ (z_1,z_2)=\bigl( (x+\rmi y)\,r_2^n(x,y)\,\rme^{\rmi n\theta_2},
r_2(x,y)\,\rme^{\rmi\theta_2}\bigr).\]

From (\ref{eqn:n-complex-leaf}) and with $c_0=c_1+\rmi c_2$,
we then see that the intersection of each complex leaf with $S^3_0$
is given by a system of equations
\begin{equation}
\label{eqn:n-leaf}
\left\{\begin{array}{rcr}
x-\log r_2(x,y) & = & -c_1,\\
\theta_2-y      & = & c_2.
\end{array}\right.
\end{equation}

Implicit differentiation of (\ref{eqn:r2}) gives
\[ \frac{\partial r_2}{\partial x}=
\frac{-xr_2^{2n-1}}{n(x^2+y^2)r_2^{2n-2}+1},\]
from which we derive with $r_2^n\leq r_2$ the estimate
\begin{equation}
\label{eqn:r2-x}
\left|\frac{\partial r_2}{\partial x}\right|\leq
r_2\,\frac{|x|r_2^{n-1}}{(xr_2^{n-1})^2+1}\leq
\frac{r_2}{2}.
\end{equation}
So the partial derivative with respect to $x$ of the function
$(x,y)\mapsto x-\log r_2(x,y)$ lies in the interval $[1/2,3/2]$,
which means that the first equation in (\ref{eqn:n-leaf}) implicitly
defines $x$ as a smooth function of~$y\in\R$ (depending on $n$ and~$c_1$).
Hence, the solution curve of (\ref{eqn:n-leaf}) is parametrised by
\[ \R\ni y\longmapsto \bigl( x(y), y,\theta_2=y+c_2\bigr),\]
which verifies the claim made earlier that the intersection
of a complex leaf with $S^3$ gives a single leaf of~$\FF_n$.

For $y\rightarrow\pm\infty$ we have
\[ \frac{\sqrt{1-r_2^2}}{r_2^n}=\left|\frac{z_1}{z_2^n}\right|
\longrightarrow\infty,\]
and hence $r_2\rightarrow 0$, which proves the proposition.
\end{proof}

Next, as for the parametric family, we describe the limiting behaviour
of the angle $\theta_1(y)$ for $y\rightarrow\pm\infty$.
From $z_1=(x+\rmi y)r_2^n\rme^{\rmi n\theta_2}$ and (\ref{eqn:n-leaf})
we have
\[ \theta_1(y)=n(c_2+y)+\arg(x(y)+\rmi y).\]
The implicit definition of $x(y)$ in (\ref{eqn:n-leaf})
and the limiting behaviour $r_2\rightarrow 0$ for $y\rightarrow\pm\infty$
entail that $x(y)\rightarrow -\infty$ for $y\rightarrow\pm\infty$.
(One may notice that $x(y)=x(-y)$, and by implicit
differentiation one sees that the function $y\mapsto x(y)$
has a single local maximum at $y=0$.) It follows that
\[ \arg(x(y)+\rmi y)\in
\begin{cases}
[\pi/2,\pi]   & \text{for $y\gg 1$},\\
[-\pi/2,-\pi] & \text{for $y\ll -1$}.
\end{cases}\]
In fact, by a more careful analysis one can show that
\[ x(y)+c_1+\frac{1}{n}\log |y|\longrightarrow 0,\]
and hence
$\arg(x(y)+\rmi y)\rightarrow\pm\pi/2$ for $y\rightarrow\pm\infty$.
Our more rough estimate, however, is sufficient to
conclude that $\theta_1(y)\rightarrow\pm\infty$ for $y\rightarrow\pm\infty$.
Geometrically this means that the Hopf circle $S^1\times\{0\}$
is the $\alpha$- and $\omega$-limit set of each leaf in $\FF_n$.

\vspace{2mm}

In order to visualise the global topology of the foliation~$\FF_n$,
we introduce an auxiliary $2$-dimensional foliation~$\EE_n$
of~$S^3$. The flow
\[ \psi_t\co(z_1,z_2)\longmapsto(\rme^{\rmi nt}z_1,\rme^{\rmi t}z_2),\;\;\;
t\in\R,\]
on $S^3$ is along the fibres of the Seifert fibration
$\pi_n\co S^3\rightarrow\hat{\C}$.
From $\psi_t^*\omega_n=\rme^{\rmi (n+1)t}\omega_n$ we see that the flow
$\psi_t$ preserves the foliation~$\FF_n$. The Hopf circle
$S^1\times\{0\}$ is mapped to itself by~$\psi_t$, but on the
complement $S^3_0$ the flow is $2\pi$-periodic and
transverse to~$\FF_n$, since
\[ \omega_n(nz_1\partial_{z_1}+z_2\partial_{z_2})=z_2^{n+1}.\]
So each leaf of $\FF_n$ in $S^3_0$ sweeps out a cylindrical surface. We
write $\EE_n$ for the singular $2$-dimensional foliation of $S^3$ made up
of these surfaces and a single $1$-dimensional leaf $S^1\times\{0\}$.
From Proposition~\ref{prop:n-asymptotic} we deduce that
the closure of each $2$-dimensional leaf of $\EE_n$ is the
union of that leaf with $S^1\times\{0\}$.

Observe that in terms of
the coordinates $(x,y,\theta_2)$ on $S^3_0$, the flow
$\psi_t$ is simply given by
\[ \psi_t\co(x,y,\theta_2)\longmapsto (x,y,\theta_2+t).\]
With the description of the leaves of $\FF_n$ in $S^3_0$
given in (\ref{eqn:n-leaf}),
this tells us that the leaves of $\EE_n$ in $S^3_0$ are
the inverse images under $\pi_n$ of the curves in $\C$ determined by
an equation
\begin{equation}
\label{eqn:pi-EE-n}
x-\log r_2(x,y)=-c_1.
\end{equation}
As $c_2$ varies in (\ref{eqn:n-leaf}), we obtain the leaves of
$\FF_n$ within a single leaf of~$\EE_n$.

The following proposition says that, up to a $C^1$-diffeomorphism,
the foliation $\EE_n$ looks homogeneous.

\begin{prop}
There is a $C^1$-diffeomorphism $\tilde{\sigma}$ of $S^3$, fixed along
$S^1\times\{0\}$ and of class $C^{\infty}$ on $S^3_0$, which sends
$\EE_n$ to the $2$-dimensional foliation of $S^3$ with a singular leaf
$S^1\times\{0\}$, and all $2$-dimensional leaves of the form
$\pi_n^{-1}(\{ x=\mathrm{const.}\})$. In other words, $\tilde{\sigma}(\EE_n)$
is the preimage under $\pi_n$ of the standard foliation of
$\hat{\C}$ with a singular point of index~$2$ at~$\infty$.
\end{prop}

\begin{proof}
We first construct a $C^1$-diffeomorphism $\sigma$ of $\hat{\C}$
that brings the foliation $\pi_n(\EE_n)$ given by (\ref{eqn:pi-EE-n})
into standard form. Set
\[ \sigma(z)=x-\log r_2(x,y)+\rmi y\;\;\text{for $z=x+\rmi y\in\C$},\;\;\;
\sigma(\infty)=\infty.\]
From the estimate (\ref{eqn:r2-x}) and the comment following it we see
that $\sigma$ maps $\C$ diffeomorphically onto itself, and it
obviously `linearises' the foliation of~$\hat{\C}$. Notice that
$\sigma(0)=0$.

To examine the differentiability of $\sigma$ near~$\infty$, we use the
coordinate $w$ on $\hat{\C}\setminus\{0\}=\C^{*}\cup\{\infty\}$ given by
$w(z)=1/z$ for $z\in\C^*$ and $w(\infty)=0$. From the implicit
definition of $r_2(z)=r_2(x,y)$ in (\ref{eqn:r2}) we have
\[ r_2^{2n}=\frac{1-r_2^2}{|z|^2}<\frac{1}{|z|^2}.\]
Feeding this estimate back into the defining equation, we obtain
\[ \frac{1-|z|^{-2/n}}{|z|^2}<r_2^{2n}<\frac{1}{|z|^2}.\]
This gives us the growth estimate
\[ \log r_2(z)=-\frac{1}{n}\log|z|+O(|z|^{-2/n})=
\frac{1}{n}\log|w|+O(|w|^{2/n})\;\;\;\text{for $w\rightarrow 0$}.\]
A straightforward calculation yields
\[ \frac{1}{\sigma(z)}=w+\frac{1}{n}w^2\log|w|+
O(|w|^{2+\frac{2}{n}})\;\;\;\text{for $w\rightarrow 0$},\]
and a similar estimate for the differential of~$\sigma$.
This means that $\sigma$ is $C^1$ near $w=0$, and its
differential admits $|w|\log |w|$ as a modulus of continuity.

Next we want to construct the diffeomorphism $\tilde{\sigma}$
of $S^3$ as a lift of $\sigma$, that is, $\tilde{\sigma}$ should satisfy the
equation $\pi_n\circ\tilde{\sigma}=\sigma\circ\pi_n$.
For this construction we use explicit coordinates on $S^3_0$ and
\[ S^3_{\infty}:=S^3\setminus\pi_n^{-1}(0)=S^3\setminus
\bigl(\{0\}\times S^1\bigr).\]
For $S^3_0$, we use the parametrisation from the proof of
Proposition~\ref{prop:n-asymptotic}:
\[ \begin{array}{rccc}
\phi_0\co & \C\times S^1            & \longrightarrow
    & S^3_0\\[1mm]
          & (z,\rme^{\rmi\theta_2}) & \longmapsto
    & (zr_2^n(z)\rme^{\rmi n\theta_2},r_2(z)\rme^{\rmi\theta_2}),
\end{array}\]
with inverse diffeomorphism given by
\[ \phi_0^{-1}\co (z_1,z_2)\longmapsto
\Bigl(\frac{z_1}{z_2^n},\frac{z_2}{|z_2|}\Bigr).\]
For the parametrisation of $S^3_{\infty}$, it is convenient to replace
$\C$ by the open unit disc $\D\subset\C$. We then define a diffeomorphism
\[ \begin{array}{rccc}
\phi_{\infty}\co & S^1\times\D               & \longrightarrow
    & S^3_{\infty}\\[1mm]
                 & (\rme^{\rmi\theta_1},z_2) & \longmapsto
    & \bigl(\sqrt{1-|z_2|^2}\,\rme^{\rmi\theta_1},z_2\bigr),
\end{array} \]
with inverse map
\[ \phi_{\infty}^{-1}\co (z_1,z_2)\longmapsto
\Bigl(\frac{z_1}{|z_1|},z_2\Bigr).\]

We first construct the lift $\tilde{\sigma}$ near $S^1\times\{0\}$,
i.e.\ near the point $\infty$ (or $w=0$) in the base. From the
growth estimate for $\log r_2$ we have near $w=0$ a well-defined
complex-valued function
\[ \mu(w):=\frac{1}{\sqrt[n]{1-w\log r_2(z)}} \]
with $\arg\mu$ close to zero, and this function admits
$|w|\log|w|$ as modulus of continuity.
For points $p\in\hat{\C}$ near $\infty$ we have
\[ w(\sigma(p))=\frac{1}{z(p)-\log r_2(z(p))}=
\frac{w(p)}{1-w(p)\log r_2(z(p))},\]
hence
\[ w(\sigma(p))=w(p)\cdot\mu(w(p))^n.\]

By slight abuse of notation, we now suppress the
parametrisations, i.e.\ we think of $\pi_n|_{S^3_{\infty}}$
as a map $S^1\times\D\rightarrow\C$, and of $\sigma$ as the
germ of a map $(\D,0)\rightarrow (\D,0)$. Then
\[ \pi_n(\rme^{\rmi\theta_1},z_2)=\frac{z_2^n}{\sqrt{1-|z_2|^2}}
\rme^{-\rmi\theta_1}=\psi(z_2)^n\rme^{-\rmi\theta_1},\]
where $\psi\co\D\rightarrow\C$ is the diffeomorphism
\[ \psi\co z\longmapsto\frac{z}{(1-|z|^2)^{1/2n}},\]
and
\[ \sigma\circ\pi_n(\rme^{\rmi\theta_1},z_2)=\psi(z_2)^n\rme^{-\rmi\theta_1}\,
\bigl(\mu(\psi(z_2)^n\rme^{-\rmi\theta_1})\bigr)^n.\]
Thus, in order to obtain a commutative diagram
\[
\begin{CD}
S^1\times\D   @>\tilde{\sigma}>>   S^1\times\D\\
@V\pi_nVV                          @VV\pi_nV\\
\C            @>\sigma>>           \C
\end{CD}
\]
with a map $\tilde{\sigma}$ defined near $S^1\times\{0\}\subset S^1\times\D$,
we can simply set
\[ \tilde{\sigma}(\rme^{\rmi\theta_1},z_2):=(\rme^{\rmi\theta_1},
\tilde{z}_2)\]
with
\[ \tilde{z}_2:=\psi^{-1}\bigl(\psi(z_2)\cdot\mu(\psi(z_2)^n
\rme^{-\rmi\theta_1})\bigr).\]
Notice that $\tilde{\sigma}$ fixes $S^1\times\{0\}$ pointwise.
Given the continuity properties of $\mu$ near $w=0$, and
the fact that the diffeomorphism $\psi$ goes like $z_2$ near $z_2=0$,
we see that $\tilde{\sigma}$ is $C^1$ at $z_2=0$, with first derivative
admitting $|z_2|\log|z_2|$ as modulus of continuity; outside
$z_2=0$ the local diffeomorphism $\tilde{\sigma}$ is smooth.

\begin{rem}
In fact one can show that $\tilde{\sigma}$ (for
a given $n$) has derivatives up to order~$n$, and the
$n^{\mathrm{th}}$ derivative admits $|z_2|\log|z_2|$ as modulus of
continuity. Since $\psi$ is a diffeomorphism, the regularity
of $\tilde{z}_2$ as a function of $z_2$ and $\theta_1$ is the same
as that of
$(\zeta,\theta)\mapsto \zeta\cdot\mu(\zeta^n\rme^{-\rmi\theta})$.
By a more careful growth estimate for $\log r_2(z)$, one obtains
the claimed result.
\end{rem}

Next we wish to construct the lift $\tilde{\sigma}$ on $S^3_0$,
that is, over $\C\subset\hat{\C}$ in the base, making sure that it
coincides with the previous construction near~$\infty$.
Again we work in coordinates, so we want to construct $\tilde{\sigma}$
such that the diagram
\[
\begin{CD}
\C\times S^1   @>\tilde{\sigma}>>   \C\times S^1\\
@V\pi_nVV                           @VV\pi_nV\\
\C             @>\sigma>>           \C
\end{CD}
\]
commutes. From the definition of $\phi_0$ we see that in this
diagram the vertical map $\pi_n$ is simply the projection onto the
first factor, so $\tilde{\sigma}$ must be of the form
\begin{equation}
\label{eqn:tilde-sigma}
\tilde{\sigma}(z,\rme^{\rmi\theta_2})=
(\sigma (z),\rme^{\rmi\tilde{\theta}_2})
\end{equation}
for a suitable function $\tilde{\theta}_2(z,\theta_2)$.

The composition
\[ S^1\times(\D\setminus\{0\})\stackrel{\phi_{\infty}}{\longrightarrow}
S^3_{\infty}\cap S^3_0 \stackrel{\phi_0^{-1}}{\longrightarrow}
(\C\setminus\{0\})\times S^1\]
is given in the second factor by $z_2\mapsto z_2/|z_2|$. It follows
that near $z=\infty$, the function $\tilde{\theta}_2$ must be given by
$\arg\tilde{z}_2$. Since the diffeomorphism $\psi$ preserves
the argument, this gives
\[ \tilde{\theta}_2(z,\theta_2)=\theta_2+\arg\bigl(\mu(\psi(z_2)^n
\rme^{-\rmi\theta_1}) \bigr),\]
where $(\rme^{\rmi\theta_1},z_2)=
\phi_{\infty}^{-1}\circ\phi_0(z,\rme^{\rmi\theta_2})$,
so we can write this as
\[ \tilde{\theta}_2(z,\theta_2)=\theta_2+f(z,\theta_2).\]

Our previous definition of the lift $\tilde{\sigma}$ near $z=\infty$
means that there $f$ is given, and it takes values close to zero.
From the coordinate description of $\tilde{\sigma}$
in (\ref{eqn:tilde-sigma}) and with $|f|$ small we see that
$\tilde{\sigma}$ maps the $S^1$-fibre over $z$ diffeomorphically
with degree~$1$ onto the fibre over $\sigma(z)$, which
necessitates $\partial f/\partial\theta_2>-1$. This is
a convex condition, so the $f$ given near $z=\infty$ can be extended
smoothly over $\C$ subject to this condition. This completes the
construction of the lift~$\tilde{\sigma}$.
\end{proof}

According to this proposition, when each of the foliations $\EE_n$
is viewed relative to the Seifert fibration~$\pi_n$, these
foliations look the same for all~$n$. In other words, the topology
of the foliation is essentially encoded in the Seifert fibration.

An alternative and more intrinsic way to understand the topology of
$\FF_n$ and $\EE_n$ is to consider surfaces of section.

\begin{prop}
For each $n\in\N$, the $2$-disc $\{\theta_2=\mathrm{const.},\,
r_1<1\}$ with
boundary the closed leaf $S^1\times\{0\}$ is a global surface of
section for the foliation $\FF_n$.
\end{prop}

\begin{proof}
We have
\[ z_1\,\rmd\oz_1+\oz_1\,\rmd z_1+z_2\,\rmd\oz_2+\oz_2\,\rmd z_2=
2(x_1\,\rmd x_1+y_1\,\rmd y_1+x_2\,\rmd x_2+ y_2\,\rmd y_2)\]
and
\[ z_2\,\rmd\oz_2-\oz_2\,\rmd z_2=-2\rmi r_2^2\,\rmd\theta_2.\]
The wedge product of these two $1$-forms with $\omega_n\wedge\oomega_n$
is a volume form on $\C^2$ multiplied by a factor
\[ n|z_1|^2|z_2|^2+|z_2|^4+|z_2|^2\re(z_1\oz_2^n),\]
which is positive on $S^3\setminus\bigl(S^1\times\{0\}\bigr)$.
This means that $\ker\omega_n$ is transverse to the disc
$\{\theta_2=\mathrm{const.},\, r_1<1\}$.
\end{proof}

More interesting is the behaviour of $\FF_n$ near the closed
leaf $S^1\times\{0\}$, so we now consider the discs
$\{\theta_1=\mathrm{const.}\}$. These discs are surfaces
of section near $r_2=0$, that is, near the closed leaf.
For the concept of Leau--Fatou flower used in the next
proposition see~\cite[\S 10]{miln06}.

\begin{prop}
\label{prop:LF}
The Poincar\'e return map of $\FF_n$ on the disc
$\{\theta_1=\mathrm{const.},\, r_2<1\}$ near the central fixed point has a
Leau--Fatou flower with $n$ attracting petals.
\end{prop}

\begin{proof}
Without loss of generality, we consider the disc $\Delta:=\{\theta_1=0,\,
r_2<1\}$,
on which we take $(r_2,\theta_2)$ as polar coordinates. The Seifert
fibres of $\pi_n$ are transverse to~$\Delta$,
hence so is the flow $\psi_t$, which implies that the leaves of
$\EE_n$ are likewise transverse. From (\ref{eqn:pi-EE-n}) we see
that the intersection of $\EE_n$ with $\Delta$ is given by
curves of the form
\[ \cos(n\theta_2)=\frac{r_2^n}{\sqrt{1-r_2^2}}(\log r_2-c_1)\]
for varying values of~$c_1$. These are shown in Figures \ref{figure:LF-1}
and \ref{figure:LF-3} for $n=1$ and $n=3$, respectively.
The centre of $\Delta$ is the intersection point with
the closed leaf $S^1\times\{0\}$ of $\FF_n$.

\begin{figure}[ht]
\centering
\includegraphics[scale=0.45]{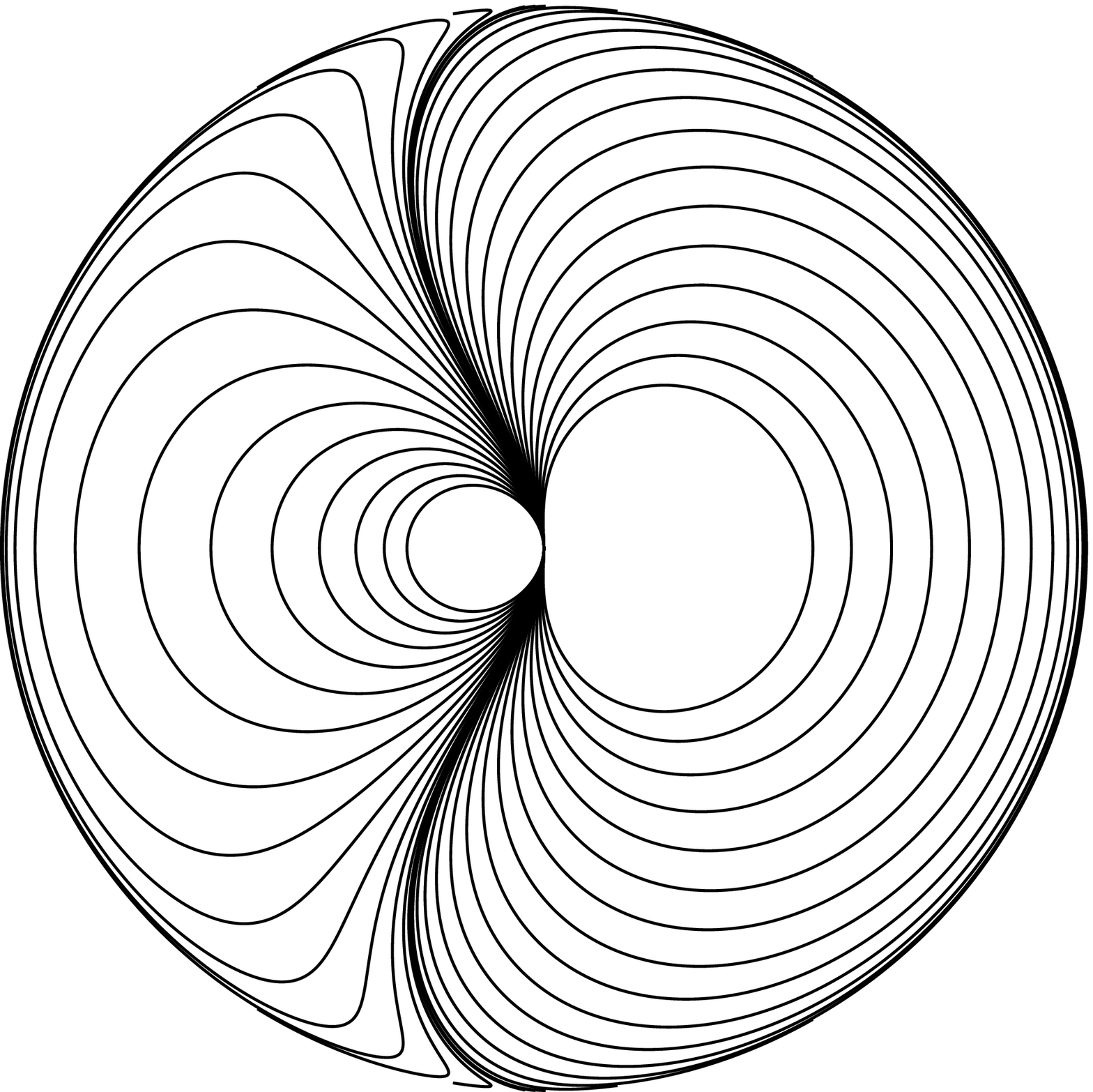}
  \caption{The foliation $\Delta\cap\EE_1$.}
  \label{figure:LF-1}
\end{figure}

\begin{figure}[ht]
\centering
\includegraphics[scale=0.45]{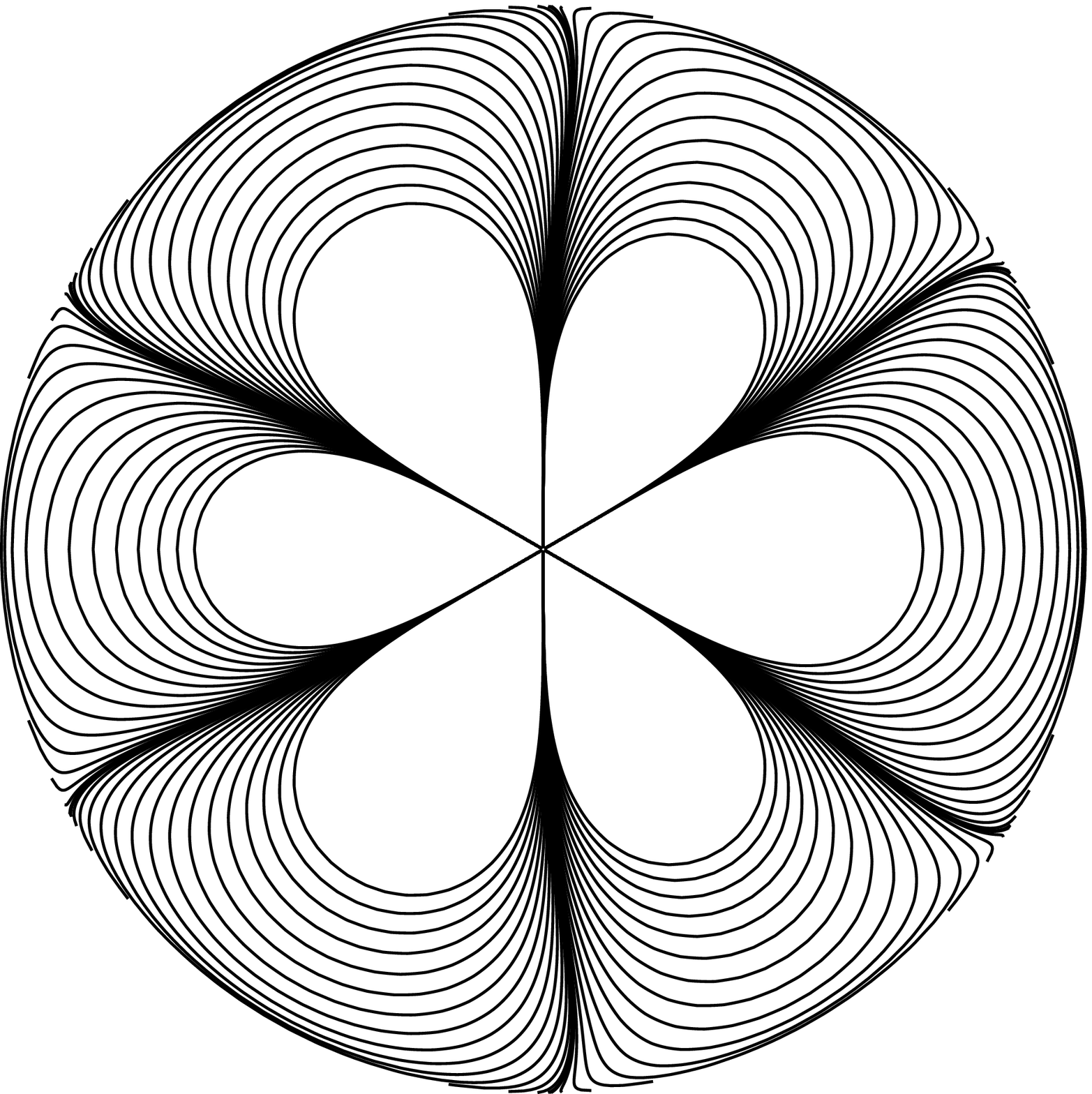}
% The eps version of the figure has
% smaller resolution than the pdf.
  \caption{The foliation $\Delta\cap\EE_3$.}
  \label{figure:LF-3}
\end{figure}

The return time for any point $p\in\Delta$ under the flow
$\psi_t$ is $t=2\pi/n$, and we have
\[ \psi_{2\pi/n}(r_2,\theta_2)=\Bigl(r_2,\theta_2+\frac{2\pi}{n}\Bigr).\]
Hence, in the picture for $n=1$, each loop (without the
central point) corresponds to the
intersection of $\Delta$ with a single leaf of~$\EE_n$; in the case
$n=3$, each cylindrical leaf $\R\times S^1$of $\EE_3$ cuts $\Delta$ in three
open loops (corresponding to the $\R$-factor) obtained from one another
by rotation through $2\pi/3$.

Each leaf of $\FF_n$ is contained in a leaf of~$\EE_n$. As we saw earlier,
the non-closed leaves of $\FF_n$ have infinite variation in
$\theta_1$-direction, and they approach $S^1\times\{0\}$ in forward and
backward time. Near the centre of~$\Delta$, where $\FF_n$
is transverse to~$\Delta$, each leaf of $\FF_n$ meets
$n$ loops of $\Delta\cap\EE_n$ in cyclic order, and in each loop
the intersection points move from one end to the other with time.
In adjacent loops, these intersection points move in
opposite direction. This means that there are open sectors of width
$2\pi/n$ where the intersection points approach the origin along
the central direction of the sector, so we have a Leau--Fatou flower
with $n$ attracting petals in the terminology of~\cite[\S 10]{miln06};
correspondingly, there are $n$ repelling petals.
\end{proof}

\begin{rem}
The results in this section show that even the simple Poincar\'e
foliations on the $3$-sphere give rise to interesting dynamical
patterns. For a more wide-ranging analysis of transversely
holomorphic foliations of codimension~$1$ from a dynamical
point of view see \cite[Chapter~6]{asuk10} and~\cite{ggs01}.
\end{rem}

We end the discussion of the topology of the foliations $\FF_n$
with the following branched cover description.

\begin{prop}
There is an $n$-fold branched cover $S^3\rightarrow S^3$, branched along
$S^1\times\{0\}$, that pulls back $\FF_1$ to~$\FF_n$.
\end{prop}

\begin{proof}
We start with the branched covering
\[ \begin{array}{rccc}
p_n\co & \C^2      & \longrightarrow & \C^2\\
       & (z_1,z_2) & \longmapsto     & (nz_1,z_2^n).
\end{array}\]
This satisfies $p_n^*\omega_1=nz_2^{n-1}\omega_n$, so it maps the
complex leaves of the foliation $\CC_n$ to those of~$\CC_1$.

Define a diffeomorphic copy of $S^3$ by
\[ \Sigma_n:=p_n^{-1}(S^3)=\{(z_1,z_2)\co n^2|z_1|^2+|z_2|^{2n}=1\}.\]
Then $p_n$ restricts to a branched covering $\Sigma_n\rightarrow S^3$.
We denote by $\FF_n'$ the $1$-dimensional foliation of $\Sigma_n$
given by the intersection with~$\CC_n$; this foliation is mapped
by $p_n$ to $\FF_1$.

It remains to construct a diffeomorphism
\[ \Phi_n\co (S^3,\FF_n)\longrightarrow (\Sigma_n,\FF_n').\]
To this end, we consider the holomorphic vector field
\[ (nz_1+z_2^n)\,\partial_{z_1}+z_2\partial_{z_2}\]
tangent to the leaves of~$\CC_n$. Its complex flow, whose
orbits are the leaves of~$\CC_n$, is given by
\[ \Psi_n^{\zeta}(z_1,z_2)=(\rme^{n\zeta}z_1+\zeta\rme^{n\zeta}z_2^n,
\rme^{\zeta}z_2).\]
Given any smooth complex-valued function $\zeta(z_1,z_2)$, the map
$\Phi_n$ defined by
\[ \Phi_n(z_1,z_2):=\Psi_n^{\zeta(z_1,z_2)}(z_1,z_2)\]
likewise preserves the leaves of~$\CC_n$; this can be seen by
geometric reasoning or with a direct computation showing
$\Phi_n^*\omega_n=\rme^{(n+1)\zeta(z_1,z_2)}\omega_n$.

We would now like to choose $\zeta$ as a real-valued function
on $S^3$ such that $\Phi_n^{\zeta(p)}(p)\in\Sigma_n$ for each
$p\in S^3$. This leads to the implicit equation
\[ n^2\rme^{2n\zeta}|z_1+\zeta z_2^n|^2+\rme^{2n\zeta}|z_2|^{2n}=1\]
for $\zeta$. A straightforward computation shows that the derivative
of the left-hand side with respect to $\zeta$ is everywhere
positive. Moreover, the left-hand side goes to zero for $\zeta\rightarrow
-\infty$, and to infinity for $\zeta\rightarrow\infty$. So this
implicit equation defines a unique smooth real-valued function $\zeta$
with the desired properties. The map $p\mapsto \Phi_n^{\zeta(p)}(p)$
then maps $S^3$ into $\Sigma_n$, and since the inverse map
can be constructed by analogous means, it is actually
a diffeomorphism.
\end{proof}
\section{Rigidity results}
In this section we discuss a number of cases where the
common kernel foliation determines the transverse
holomorphic structure or the taut contact circle.

\begin{lem}
Let $\omegac=\omega_1+\rmi\omega_2$ be a formally integrable complex $1$-form.
Let $Y$ be a vector field generating the common kernel foliation,
and write $L_Y\omegac=(f+\rmi g)\omegac$ with real-valued
functions $f$ and~$g$. Then $\omega_1,\omega_2$ are contact forms
(and hence define a taut contact circle) precisely on the open set
where $g\neq 0$.
\end{lem}

\begin{proof}
We compute
\begin{eqnarray*}
\rmi Y\ip(\omega_1\wedge\rmd\omegac)
 & = & -\rmi\omega_1\wedge L_Y\omegac\\
 & = & -\rmi\omega_1\wedge (f+\rmi g)\omegac\\
 & = & (f+\rmi g)\omega_1\wedge\omega_2.
\end{eqnarray*}
Taking the imaginary part, we find
\[ Y\ip(\omega_1\wedge\rmd\omega_1)=g\omega_1\wedge\omega_2.\]
This means
\[ \omega_1\wedge\rmd\omega_1=g\,\rmd V,\]
where $\rmd V$ is the volume form defined by $Y\ip\rmd V=
\omega_1\wedge\omega_2$.
\end{proof}

We retain the definition of $Y$, $g$ and $\rmd V$ for the
next lemma and its proof, as well as the theorem
that follows.

\begin{lem}
Let $\omegac'=\omegac+\phi\oomegac$ with $|\phi|<1$
be any other $1$-form defining the same cooriented $1$-dimensional
foliation as~$\omegac$. The condition for $\omegac'$ to be
formally integrable is
\[ Y\phi=2\rmi g\phi.\]
This condition implies $Y|\phi|^2=0$, i.e.\ $|\phi|$ is constant
along the leaves of the foliation.
\end{lem}

\begin{proof}
We compute
\begin{eqnarray*}
\omegac'\wedge\rmd\omegac'
 & = & (\omegac+\phi\oomegac)\wedge(\rmd\omegac+\rmd\phi\wedge\oomegac
       +\phi\,\rmd\oomegac)\\
 & = & \rmd\phi\wedge\oomegac\wedge\omegac
       +\phi(\omegac\wedge\rmd\oomegac+\oomegac\wedge\rmd\omegac)\\
 & = & 2\rmi\,\rmd\phi\wedge\omega_1\wedge\omega_2+4\phi g\,\rmd V\\
 & = & 2\bigl(\rmi Y\phi+2\phi g\bigr)\,\rmd V,
\end{eqnarray*}
from which the integrability condition follows.

From
\[ Y|\phi|^2=(Y\phi)\ophi+\phi(Y\ophi)\]
we deduce $Y|\phi|^2=0$ if the integrability condition holds.
\end{proof}

\begin{thm}
\label{thm:unique-hol}
Each of the foliations $\FF^a$, $a\in\C\setminus\R$, and $\FF_n$, $n\in\N$,
admits a unique transverse holomorphic structure for the given
coorientation.
\end{thm}

\begin{proof}
In the notation of the two preceding lemmata, we need to show
$\phi=0$ if $\omegac$ equals one of the $\omega^a$, $a\in\C\setminus\R$,
or an $\omega_n$ (provided $\phi$ defines another formally integrable
$1$-form).

By the results in Section~\ref{section:topology}, in these foliations
all leaves (except for the second Hopf circle $\{0\}\times S^1$
in $\FF^a$) are asymptotic in at least one direction to the Hopf
circle $S^1\times\{0\}$. It follows that $|\phi|$,
being constant along the leaves, must be constant on~$S^3$.

If $|\phi|$ were non-zero, we could define a map
\[ \phi_1:=\frac{\phi}{|\phi|}\co S^3\longrightarrow S^1\subset\C,\]
still satisfying the integrability condition
$Y\phi_1=2\rmi g\phi_1$ from the foregoing lemma.
But the $\omega_n$ define contact circles,
and so does $\omega^a$ near at least one Hopf circle $O$ by
Remark~\ref{rem:a-cc}, so there we have $g\neq 0$.
This implies that $\phi_1|_O\co S^1\equiv O\rightarrow S^1$
has non-zero degree, but it also extends as a map over the Seifert disc
of~$O$. This contradiction shows that we must have $\phi=0$.
\end{proof}

\begin{rem}
For the $\FF^a$ with $a\in (0,1)$, the transverse holomorphic
structure is not unique:
\begin{itemize}
\item[-] If $a$ is rational, $\FF^a$ defines a Seifert fibration,
and different holomorphic structures on the quotient orbifold
give us different transverse holomorphic structures.
\item[-] If $a$ is irrational, the leaves still lie on Hopf tori,
and by changing the metric structure in the direction
orthogonal to the Hopf tori we obtain different transverse
conformal (and hence holomorphic) structures.
\end{itemize}

We expand a bit on the second point. Outside the Hopf circles,
the tangent bundle of $S^3$ is trivialised by the orthonormal frame
(with respect to the standard metric)
\[ \left\{\begin{array}{rcl}
\partial_{\theta_1}/r_1             & = &
            (x_1\partial_{y_1}-y_1\partial_{x_1})/r_1\\[2mm]
\partial_{\theta_2}/r_2             & = &
            (x_2\partial_{y_2}-y_2\partial_{x_2})/r_2\\[2mm]
r_2\partial_{r_1}-r_1\partial_{r_2} & = &
   \displaystyle{\frac{r_2}{r_1}\,(x_1\partial_{x_1}+y_1\partial_{y_1})-
                 \frac{r_1}{r_2}\,(x_2\partial_{x_2}+y_2\partial_{y_2})}.
\end{array}\right.
\]
The third vector in this frame is invariant under the flow
of $\partial_{\theta_1}$ and $\partial_{\theta_2}$. Any metric
for which the first two vectors fields are orthonormal, and
the third one orthogonal with length a function of~$r_1$, defines
a transverse conformal structure for $\FF^a$, $a\in (0,1)$.
\end{rem}

The following corollary improves on Corollary~\ref{cor:P-fol};
we do not need to know the transverse holomorphic structure
to determine $\FF^a$. Recall that a Poincar\'e foliation belongs
to the parametric family if and only if it has at least two
closed leaves.

\begin{cor}
\label{cor:P-recover}
From any cooriented Poincar\'e foliation $\FF$ in
the parametric family (but without
any \emph{a priori} given transverse holomorphic
structure) one can recover
the value $a(1-a)$ --- and hence the class $[a]\in\PP/(a\sim 1-a)$ ---
for which there is an orientation-preserving diffeomorphism
of $S^3$ sending $\FF$ to $\FF^a$ as a cooriented foliation.
\end{cor}

\begin{proof}
We need to show that $\FF^a$ determines~$a(1-a)$. If $a\in\C\setminus\R$,
then $\FF^a$ admits a unique transverse holomorphic structure,
and the Bott invariant of this structure gives us
$a(1-a)$ by Proposition~\ref{prop:GV-Poincare}.
If $a\in (0,1)$, then by Remark~\ref{rem:alpha-a}
we are in the situation of Theorem~\ref{thm:GV-real}. Thus,
although there is a choice of transverse holomorphic structures,
they all yield the same Bott invariant as~$\omega^a$,
and again we recover $a(1-a)$.
\end{proof}

From this we now want to deduce the uniformisation result that the moduli
space of conformal structures on any orbifold $S^2(k_1,k_2)$, where
$k_1,k_2\in\N$ are not necessarily coprime,
is a single point. This class of $2$-dimensional
orbifolds contains all the bad ones, i.e.\ those not covered
by a surface:  tear-drops, where precisely one of the $k_i$ is equal to
one, and asymmetric spindles, where $k_1,k_2$ are different and both
greater than~$1$. We begin with a topological preparation.

\begin{prop}
\label{prop:Seifert-S-L}
Given any natural numbers $k_1,k_2$, there are coprime natural numbers
$p_1,p_2$ and a natural number $m$ such that the Seifert fibration
of $S^3\subset\C^2$ determined by the $S^1$-action
\[ \theta(z_1,z_2)=(\rme^{\rmi p_1\theta}z_1,\rme^{\rmi p_2\theta}z_2),\]
which has base orbifold $S^2(p_1,p_2)$, descends to a Seifert fibration
of the left-quotient
\[ L(m,m-1)=S^3/(z_1,z_2)\sim (\rme^{2\pi\rmi/m}z_1,\rme^{-2\pi\rmi/m}z_2) \]
with base orbifold $S^2(k_1,k_2)$. For $p_1,p_2$ one may always take
the pair of coprime natural numbers with $p_1/p_2=k_1/k_2$, and
$m=k_1+k_2$.
\end{prop}

\begin{proof}
In the described Seifert fibration of~$S^3$, the regular fibres
have length $2\pi$, and the multiple fibres through $(1,0)$ and $(0,1)$
have length $2\pi/p_1$ and $2\pi/p_2$, respectively. The $\Z_m$-action
on $S^3$ commutes with the $S^1$-action, so it sends Seifert fibres
to Seifert fibres and induces the structure of a Seifert fibration
on $L(m,m-1)$.

The two multiple fibres in $S^3$ are mapped into
themselves by the $\Z_m$-action, so the length of the corresponding
fibres in $L(m,m-1)$ is $2\pi/p_1m$ and $2\pi/p_2m$, respectively.
The length of the regular Seifert fibres in $L(m,m-1)$ is given
by the minimal $\theta\in (0,2\pi]$ such that there are natural numbers
\[ k\in\{1,2,\ldots,m\},\;\;\;
l_1\in\{0,1,\ldots,p_1-1\},\;\;\;
l_2\in\{1,2,\ldots,p_2\} \]
with
\begin{equation}
\label{eqn:theta}
\left\{\begin{array}{rcr}
p_1\theta & = & \displaystyle{2\pi\frac{k}{m}+2\pi l_1},\\[3mm]
p_2\theta & = & \displaystyle{-2\pi\frac{k}{m}+2\pi l_2}.
\end{array}\right.
\end{equation}
This implies $(p_1+p_2)\theta=2\pi(l_1+l_2)$. Hence, the minimal
$\theta$ is $2\pi/(p_1+p_2)$, which can indeed be realised
for a suitable $k$ if $m$ is a multiple of $p_1+p_2$.

Now, given $k_1,k_2$, set $m=k_1+k_2$ and let $p_1,p_2$ be the coprime
natural numbers with $k_1/k_2=p_1/p_2$. Then (\ref{eqn:theta})
is satisfied with $\theta=2\pi/(p_1+p_2)$, $l_1=0$,
$l_2=1$, and $k=k_1$. So the regular fibres in $L(m,m-1)$
have length
\[ \frac{2\pi}{p_1+p_2}=\frac{2\pi k_1}{p_1(k_1+k_2)}=
\frac{2\pi k_2}{p_2(k_1+k_2)},\]
compared to the length of the multiple fibres
\[ \frac{2\pi}{p_jm}=\frac{2\pi}{p_j(k_1+k_2)},\;\;\; j=1,2,\]
which means that the multiplicities are $k_1,k_2$.
\end{proof}

\begin{rem}
The choice of $m=k_1+k_2$ is not the smallest possible, in
general. For instance, if $k_1=p_1^2$ and $k_2=p_1p_2$ with
$p_1,p_2$ coprime, one can take $m=p_1$, since the corresponding
$\Z_m$-action freely permutes the regular fibres in~$S^3$.
\end{rem}

In the following uniformisation theorem and its proof it
is convenient to think of a conformal structure on an orbifold
as a transverse conformal structure on a Seifert fibration over it,
and of an orbifold diffeomorphism as a fibre-preserving
diffeomorphism of that Seifert manifold.
This uniformisation theorem has been proved previously by Zhu~\cite{zhu97},
using the Ricci flow.

\begin{thm}
\label{thm:orbi}
For any natural numbers $k_1,k_2$, the conformal
structure on the orbifold $S^2(k_1,k_2)$ is unique up to orbifold
diffeomorphism.
\end{thm}

\begin{proof}
Define the coprime natural numbers $p_1,p_2$ by the condition
$p_1/p_2=k_1/k_2$.
Consider the diagram
\[
\begin{CD}
S^3       @>>>   L(m,m-1)\\
@VVV             @VVV\\
S^2(p_1,p_2)  @>>>  S^2(k_1,k_2)
\end{CD}
\]
from the discussion in the preceding proposition.
Choose a contact form $\omega_1$ on $L(m,m-1)$ for which the Seifert
fibration $L(m,m-1)\rightarrow S^2(k_1,k_2)$ is Legendrian, i.e.\ tangent
to $\ker\omega_1$. For instance, the $1$-form $\omega_1^a$
on $S^3$ with $a/(1-a)=p_1/p_2$ is such a contact form
on the Seifert fibration $S^3\rightarrow S^2(p_1,p_2)$,
and being $Z_m$-invariant it descends to $L(m,m-1)$.

Given a conformal structure on $S^2(k_1,k_2)$,
define a second $1$-form $\omega_2$ on the lens space
$L(m,m-1)$ by stipulating that
the $2$-plane field $\ker\omega_2$ be
tangent to the fibres of $L(m,m-1)\rightarrow S^2(k_1,k_2)$,
and that $\omega_1\otimes\omega_1+\omega_2\otimes\omega_2$ define the
transverse conformal structure; this $\omega_2$ is unique up to sign.
Then $\omegac:=\omega_1+\rmi\omega_2$ is formally
integrable. With $\omega_1$ being a contact form, this implies
that $(\omega_1,\omega_2)$ is in fact a taut contact
circle.

By the classification of
taut contact circles in \cite[Proposition~6.1]{gego95},
$(\omega_1,\omega_2)$ equals $(\omega_1^a,\omega_2^a)$
(regarded as taut contact circle on $L(m,m-1)$)
up to homothety and diffeomorphism for a unique~$[a]$.
By Corollary~\ref{cor:P-recover}, this must be the
class $[a]$ determined by $a/(1-a)=p_1/p_2$, that is,
the one we chose above to define~$\omega_1$.
Thus, the given conformal structure on $S^2(k_1,k_2)$ is diffeomorphic
to the one determined by $(\omega_1^a,\omega_2^a)$
on $L(m,m-1)$.
\end{proof}

For taut contact circles we have an even more succinct statement
than Corollary~\ref{cor:P-recover}.

\begin{thm}
\label{thm:common-kernel}
The homothety class of a taut contact circle on $S^3$ (inducing
the standard orientation) is determined, up to
orientation-preserving diffeomorphism, by its cooriented
common kernel foliation.
\end{thm}

\begin{proof}
If the common kernel foliation has only one closed leaf,
the taut contact circle comes from the discrete
family $\{\omega_n\co n\in\N\}$. By Proposition~\ref{prop:log-n},
the value of $n$ can be recovered from the logarithmic monodromy
of the closed leaf. Alternatively, by
Proposition~\ref{prop:LF}, $n$ can
be read off as the number of petals in the Leau--Fatou flower
of the Poincar\'e return map.

If the common kernel foliation has more than one closed leaf,
the taut contact circle comes from the parametric family
$\{\omega^a\co [a]\in\MM\}$. Corollary~\ref{cor:P-recover}
tells us how to recover $[a]$ from the cooriented foliation.
\end{proof}

\begin{rem}
In the case of the parametric family, we may appeal alternatively
to our topological considerations. The following cases cover
all eventualities, but they are not mutually exclusive.
\begin{itemize}
\item[(i)] If the foliation defines a Seifert fibration with two
singular fibres of multiplicity $p_1,p_2$ (one or both of which
may be equal to~$1$), we determine the unordered pair
\[ \frac{a}{1-a},\;\frac{1-a}{a}\]
from Proposition~\ref{prop:Seifert}.
\item[(ii)] If the leaves foliate tori, that pair of numbers
can be read off from the slope of these foliations by
Proposition~\ref{prop:P-leaves}.
\item[(iii)] If there are only two closed leaves, we recover that
pair of numbers from their logarithmic monodromy, using
Proposition~\ref{prop:log-a}.
\end{itemize}
That pair of numbers determines $a(1-a)$ via
\[ \frac{a}{1-a}+\frac{1-a}{a}=\frac{1}{a(1-a)}-2.\]
\end{rem}

\begin{ack}
We thank Otto van Koert for producing Figures \ref{figure:LF-1}
and~\ref{figure:LF-3}.
\end{ack}

\end{document}